\documentclass{siamart1116}
\usepackage{soul} 
\usepackage{amsfonts}
\usepackage{graphicx}
\ifpdf
  \DeclareGraphicsExtensions{.eps,.pdf,.png,.jpg}
\else
  \DeclareGraphicsExtensions{.eps}
\fi
\usepackage[mathscr]{eucal}
\usepackage{amssymb}
\usepackage{enumerate}
\usepackage{cases}
\usepackage{pgf,tikz}
\usepackage{color,ulem}
\usepackage{csquotes}
\usepackage{stmaryrd}
\usepackage{subcaption}

\usepackage[notref,notcite]{showkeys}

\usepackage[show]{ed}
\definecolor{dhcol}{rgb}{0,0.5,0}

\definecolor{sccol}{rgb}{0,0,0.5}

\definecolor{amcol}{rgb}{0.5,0,0}

\newcommand{\dsp}{\displaystyle}

\newcommand{\dgreen}{\color{green!45!black}}

\numberwithin{theorem}{section}

\newcommand{\TheTitle}{On the half-space matching method for real wavenumber}
\newcommand{\TheAuthors}{A.-S. Bonnet-Ben Dhia et al.}

\headers{\TheTitle}{\TheAuthors}

\title{{\TheTitle}\thanks{Submitted to the editors DATE.
}}

\author{Anne-Sophie Bonnet-Ben Dhia\thanks{POEMS (CNRS-INRIA-ENSTA Paris), Institut Polytechnique de Paris, Palaiseau, France (\email{anne-sophie.bonnet-bendhia@ensta-paris.fr}, \email{sonia.fliss@ensta-paris.fr})} \and 
  Simon N. Chandler-Wilde\thanks{Department of Mathematics and Statistics, University of Reading, Whiteknights PO Box 220, Reading RG6 6AX, UK
    (\email{s.n.chandler-wilde@reading.ac.uk})}
    \and Sonia Fliss\footnotemark[2]
}

\ifpdf
\hypersetup{
  pdftitle={\TheTitle},
  pdfauthor={\TheAuthors}
}
\fi

\graphicspath{{Figs/}}
\begin{document}

\newcommand{\rf}[1]{(\ref{#1})}
\newcommand{\mmbox}[1]{\fbox{\ensuremath{\displaystyle{ #1 }}}}	
\newcommand{\hs}[1]{\hspace{#1mm}}
\newcommand{\vs}[1]{\vspace{#1mm}}
\newcommand{\ri}{{\mathrm{i}}}
\newcommand{\re}{{\mathrm{e}}}
\newcommand{\rd}{\mathrm{d}}
\newcommand{\R}{\mathbb{R}}
\newcommand{\Q}{\mathbb{Q}}
\newcommand{\N}{\mathbb{N}}
\newcommand{\Z}{\mathbb{Z}}
\newcommand{\C}{\mathbb{C}}
\newcommand{\K}{{\mathbb{K}}}
\newcommand{\cA}{\mathcal{A}}
\newcommand{\cB}{\mathcal{B}}
\newcommand{\cC}{\mathcal{C}}
\newcommand{\cS}{\mathcal{S}}
\newcommand{\cD}{\mathcal{D}}
\newcommand{\cH}{\mathcal{H}}
\newcommand{\cI}{\mathcal{I}}
\newcommand{\cItilde}{\tilde{\mathcal{I}}}
\newcommand{\cIhat}{\hat{\mathcal{I}}}
\newcommand{\cIcheck}{\check{\mathcal{I}}}
\newcommand{\cIstar}{{\mathcal{I}^*}}
\newcommand{\cJ}{\mathcal{J}}
\newcommand{\cM}{\mathcal{M}}
\newcommand{\cP}{\mathcal{P}}
\newcommand{\cV}{{\mathcal V}}
\newcommand{\cW}{{\mathcal W}}
\newcommand{\scrD}{\mathscr{D}}
\newcommand{\scrS}{\mathscr{S}}
\newcommand{\scrJ}{\mathscr{J}}
\newcommand{\sD}{\mathsf{D}}
\newcommand{\sN}{\mathsf{N}}
\newcommand{\sS}{\mathsf{S}}
 \newcommand{\sT}{\mathsf{T}}
 \newcommand{\sH}{\mathsf{H}}
 \newcommand{\sI}{\mathsf{I}}
\newcommand{\bs}[1]{\mathbf{#1}}
\newcommand{\bb}{\mathbf{b}}
\newcommand{\bd}{\mathbf{d}}
\newcommand{\bn}{\mathbf{n}}
\newcommand{\bp}{\mathbf{p}}
\newcommand{\bP}{\mathbf{P}}
\newcommand{\bv}{\mathbf{v}}

\newcommand{\bxi}{\boldsymbol{\xi}}
\newcommand{\boldeta}{\boldsymbol{\eta}}	
\newcommand{\ts}{\tilde{s}}
\newcommand{\tGamma}{{\tilde{\Gamma}}}
 \newcommand{\tbx}{\tilde{\bx}}
 \newcommand{\tbd}{\tilde{\bd}}
 \newcommand{\txi}{\xi}
\newcommand{\done}[2]{\dfrac{d {#1}}{d {#2}}}
\newcommand{\donet}[2]{\frac{d {#1}}{d {#2}}}
\newcommand{\pdone}[2]{\dfrac{\partial {#1}}{\partial {#2}}}
\newcommand{\pdonet}[2]{\frac{\partial {#1}}{\partial {#2}}}
\newcommand{\pdonetext}[2]{\partial {#1}/\partial {#2}}
\newcommand{\pdtwo}[2]{\dfrac{\partial^2 {#1}}{\partial {#2}^2}}
\newcommand{\pdtwot}[2]{\frac{\partial^2 {#1}}{\partial {#2}^2}}
\newcommand{\pdtwomix}[3]{\dfrac{\partial^2 {#1}}{\partial {#2}\partial {#3}}}
\newcommand{\pdtwomixt}[3]{\frac{\partial^2 {#1}}{\partial {#2}\partial {#3}}}
\newcommand{\bnabla}{\boldsymbol{\nabla}}
\newcommand{\dive}{\boldsymbol{\nabla}\cdot}
\newcommand{\curl}{\boldsymbol{\nabla}\times}
\newcommand{\Phixy}{\Phi(\bx,\by)}
\newcommand{\PhiOxy}{\Phi_0(\bx,\by)}
\newcommand{\dxPhixy}{\pdone{\Phi}{n(\bx)}(\bx,\by)}
\newcommand{\dyPhixy}{\pdone{\Phi}{n(\by)}(\bx,\by)}
\newcommand{\dxPhiOxy}{\pdone{\Phi_0}{n(\bx)}(\bx,\by)}
\newcommand{\dyPhiOxy}{\pdone{\Phi_0}{n(\by)}(\bx,\by)}
\newcommand{\eps}{\varepsilon}
\newcommand{\real}[1]{{\rm Re}\left[#1\right]} 
\newcommand{\im}[1]{{\rm Im}\left[#1\right]}
\newcommand{\ol}[1]{\overline{#1}}
\newcommand{\ord}[1]{\mathcal{O}\left(#1\right)}
\newcommand{\oord}[1]{o\left(#1\right)}
\newcommand{\Ord}[1]{\Theta\left(#1\right)}
\newcommand{\hsnorm}[1]{||#1||_{H^{s}(\bs{R})}}
\newcommand{\hnorm}[1]{||#1||_{\tilde{H}^{-1/2}((0,1))}}
\newcommand{\norm}[2]{\left\|#1\right\|_{#2}}
\newcommand{\normt}[2]{\|#1\|_{#2}}
\newcommand{\on}[1]{\Vert{#1} \Vert_{1}}
\newcommand{\tn}[1]{\Vert{#1} \Vert_{2}}
\newcommand{\xt}{\mathbf{x},t}
\newcommand{\PhiF}{\Phi_{\rm freq}}
\newcommand{\cone}{{c_{j}^\pm}}
\newcommand{\ctwo}{{c_{2,j}^\pm}}
\newcommand{\cthree}{{c_{3,j}^\pm}}
\newtheorem{thm}{Theorem}[section]
\newtheorem{lem}[theorem]{Lemma}
\newtheorem{defn}[theorem]{Definition}
\newtheorem{prop}[theorem]{Proposition}
\newtheorem{cor}[theorem]{Corollary}
\newtheorem{rem}[theorem]{Remark}
\newtheorem{conj}[theorem]{Conjecture}
\newtheorem{ass}[theorem]{Assumption}
\newtheorem{example}[theorem]{Example} 
\newcommand{\tH}{\widetilde{H}}
\newcommand{\Hze}{H_{\rm ze}} 	
\newcommand{\uze}{u_{\rm ze}}		
\newcommand{\dimH}{{\rm dim_H}}
\newcommand{\dimB}{{\rm dim_B}}
\newcommand{\IntClosOm}{\mathrm{int}(\overline{\Omega})}
\newcommand{\IntClosOmOne}{\mathrm{int}(\overline{\Omega_1})}
\newcommand{\IntClosOmTwo}{\mathrm{int}(\overline{\Omega_2})}
\newcommand{\Ccomp}{C^{\rm comp}}
\newcommand{\tCcomp}{\tilde{C}^{\rm comp}}
\newcommand{\uC}{\underline{C}}
\newcommand{\utC}{\underline{\tilde{C}}}
\newcommand{\oC}{\overline{C}}
\newcommand{\otC}{\overline{\tilde{C}}}
\newcommand{\capcomp}{{\rm cap}^{\rm comp}}
\newcommand{\Capcomp}{{\rm Cap}^{\rm comp}}
\newcommand{\tcapcomp}{\widetilde{{\rm cap}}^{\rm comp}}
\newcommand{\tCapcomp}{\widetilde{{\rm Cap}}^{\rm comp}}
\newcommand{\hcapcomp}{\widehat{{\rm cap}}^{\rm comp}}
\newcommand{\hCapcomp}{\widehat{{\rm Cap}}^{\rm comp}}
\newcommand{\tcap}{\widetilde{{\rm cap}}}
\newcommand{\tCap}{\widetilde{{\rm Cap}}}
\newcommand{\ccap}{{\rm cap}}
\newcommand{\ucap}{\underline{\rm cap}}
\newcommand{\uCap}{\underline{\rm Cap}}
\newcommand{\cCap}{{\rm Cap}}
\newcommand{\ocap}{\overline{\rm cap}}
\newcommand{\oCap}{\overline{\rm Cap}}
\DeclareRobustCommand
{\mathringbig}[1]{\accentset{\smash{\raisebox{-0.1ex}{$\scriptstyle\circ$}}}{#1}\rule{0pt}{2.3ex}}
\newcommand{\cirH}{\mathringbig{H}}
\newcommand{\cirHs}{\mathringbig{H}{}^s}
\newcommand{\cirHt}{\mathringbig{H}{}^t}
\newcommand{\cirHm}{\mathringbig{H}{}^m}
\newcommand{\cirHzero}{\mathringbig{H}{}^0}
\newcommand{\deO}{{\partial\Omega}}
\newcommand{\OO}{{(\Omega)}}
\newcommand{\Rn}{{(\R^n)}}
\newcommand{\Id}{{\mathrm{Id}}}
\newcommand{\gap}{\mathrm{Gap}}
\newcommand{\ggap}{\mathrm{gap}}
\newcommand{\isom}{{\xrightarrow{\sim}}}
\newcommand{\half}{{1/2}}
\newcommand{\mhalf}{{-1/2}}
\newcommand{\inter}{{\mathrm{int}}}

\newcommand{\Hsp}{H^{s,p}}
\newcommand{\Htq}{H^{t,q}}
\newcommand{\tHsp}{{{\widetilde H}^{s,p}}}
\newcommand{\SP}{\ensuremath{(s,p)}}
\newcommand{\Xsp}{X^{s,p}}

\newcommand{\dd}{{d}}\newcommand{\pp}{{p_*}}

\newcommand{\Rnn}{\R^{n_1+n_2}}
\newcommand{\Tr}{{\mathrm{Tr}}}
\newcommand{\sO}{\mathsf{O}}
\newcommand{\sC}{\mathsf{C}}
\newcommand{\sA}{\mathsf{A}}
\newcommand{\sM}{\mathsf{M}}
\newcommand{\sF}{\mathsf{F}}
\newcommand{\sG}{\mathsf{G}}
\newcommand{\mS}{\Gamma}
\newcommand{\omS}{{\overline{\mS}}}
\renewcommand{\O}{{\mathcal O}}
\newcommand{\NO}{{\mathcal N}}
\newcommand*{\ZZ}{{\mathbb{Z}}}
\newcommand{\bsx}{\boldsymbol{x}}
\newcommand{\bsy}{\boldsymbol{y}}
\newcommand{\bsz}{\boldsymbol{z}}
\newcommand{\bsw}{\boldsymbol{w}}
\newcommand{\bx}{\bsx}
\newcommand{\by}{\bsy}
\newcommand{\bz}{\bsz}
\maketitle

\begin{abstract}
The Half-Space Matching (HSM) method has recently been developed as a new method for the solution of 2D scattering problems with complex backgrounds, providing an alternative to {Perfectly Matched Layers} (PML) or other artificial boundary conditions. Based on  half-plane representations for the solution, the scattering problem is rewritten as a system coupling (1) a standard finite element discretisation localised around the scatterer and (2) integral equations whose unknowns are traces of the solution on the boundaries of a finite number of  overlapping half-planes contained in the domain.  While satisfactory numerical results have been obtained for real wavenumbers,  well-posedness and equivalence  of this HSM formulation to the original scattering problem have been established only for complex wavenumbers. In the present paper we show, in the case of a homogeneous background, that the HSM formulation is equivalent to the original scattering problem also for real wavenumbers, and so is well-posed, provided the traces satisfy radiation conditions at infinity analogous to the standard Sommerfeld radiation condition.  {As a key component of our argument we show that, if the trace on the boundary of a half-plane satisfies our new radiation condition, then the corresponding solution to the half-plane Dirichlet problem satisfies the Sommerfeld radiation condition in a slightly smaller half-plane.} We expect that this last result will be of independent interest, in particular in studies of rough surface scattering.
\end{abstract}



\section{Introduction and the scattering problem}
\label{sec-introduction}
\subsection{The HSM method} Recently a new method, called the Half-Space Matching (HSM) method, has been developed as an (exact) artificial boundary condition for two-dimensional time-harmonic  scattering problems. 
This method is based on explicit or semi-explicit expressions for the outgoing solutions of radiation problems in half-planes, these expressions established by using Fourier, generalized Fourier, or Floquet transforms when the background is, respectively, homogeneous \cite{BB-Fliss-Tonnoir-2018, Bon-Fli-Tja-2019} (and possibly anisotropic \cite{tonnoir2015,Bar-Bon-Fli-Tja-2018,tjandrawidjaja2019}), stratified {\cite{Ott}}, or periodic \cite{Fli-Jol-2009,Fliss:2010}. The domain exterior to a bounded region enclosing the scatterers is covered by a finite number $\mathcal{N}$ of half-planes (at least three). The unknowns of the formulation are the traces $\varphi^1,\ldots,\varphi^\mathcal{N}$ of the solution on the boundaries of these half-planes and the restriction of the solution to the bounded region. The system of equations which couples these unknowns is derived by writing compatibility conditions between the different representations of the solution. This {coupled} system includes second-kind integral equations on the infinite boundaries of the half-planes.

This new formulation is attractive and versatile as a method to truncate computational domains in  problems of scattering by localized inhomogeneities in complex backgrounds, including backgrounds that may be different at infinity in different directions. It has been employed successfully in numerical implementations {for various applications, like
optical waveguides {(including cases with different stratifications in different parts of the background domain)} \cite{Ott}, or ultrasonic non-destructive testing (with an anisotropic elastic background{\dgreen)} \cite{tonnoir2015,tjandrawidjaja2019}.}

Up to now the theoretical and numerical analysis of the method has remained an open question in the challenging, and practically relevant, non-dissipative case when waves radiate out to infinity. But a rather complete analysis has been carried out in the {simpler} dissipative case, when the solution (and its traces) decay exponentially at infinity. In that case the analysis can be done using an $L^2$ framework for the traces: the associated formulation has been shown to be of Fredholm type and well-posed in a number of cases where the background is homogeneous (but not necessarily isotropic) \cite{BB-Fliss-Tonnoir-2018,Bon-Fli-Tja-2019}, with the sesquilinear form of the weak formulation coercive plus compact, enabling the numerical analysis of the method \cite{Bon-Fli-Tja-2019}.
This analysis fails in the non-dissipative case, not least because of  the slow decay at infinity of the solution which results in non-$L^2$ traces. 

\subsection{Our main results} \label{sec:main results} In this paper we address well-posedness of the HSM formulation in the non-dissipative case for the scalar Helmholtz equation  when the background is homogeneous and isotropic, so that
\begin{equation} \label{eq:he}
-\Delta u - k^2 u = 0
\end{equation}
outside some ball. (Here, assuming $\re^{-\ri \omega t}$ time dependence with $\omega>0$, $k=\omega/c>0$ and $c>0$ are the constant wavenumber and wave speed, respectively, outside the ball.) For such configurations, it is well known that to achieve a well-posed scattering problem the scattered field $u$ must satisfy the Sommerfeld radiation condition
\begin{equation}\label{eq:Sommerfeld}
		 \frac{\partial u(\bsx)}{\partial r} - \ri k u(\bsx) = o\left(r^{-1/2}\right)\quad\text{as}\;r:=|\bsx|\rightarrow +\infty,
\end{equation}	
uniformly with respect to $\widehat \bsx := \bsx/r$. Further, it is well known that, if $u$ satisfies \eqref{eq:he} outside some ball and \eqref{eq:Sommerfeld}, then
\begin{equation} \label{eq:ffp}
u(\bsx) = \frac{\re^{\ri k r}}{r^{1/2}}\left(F(\widehat\bsx) + O(r^{-1})\right), \quad \mbox{as} \quad r\to \infty,
\end{equation}
uniformly in $\widehat\bsx := \bsx/r$, where $F\in C^\infty(S^1)$, with $S^1$ the unit circle, is the so-called {\it far-field pattern} (e.g., \cite[Lemma 2.5]{ChGrLaSp:11}).

The main results of this paper (Theorems \ref{thm:HSMMUnique} and \ref{thm:HSMMUnique_casgeneral}) are to establish  well-posedness for the HSM formulations,  stated in detail as \eqref{eq:HSMM_real} and \eqref{eq:HSMM_casgeneral}  below, in the case that $k>0$ and the background is homogeneous and isotropic, so that the scattered field $u$ satisfies \eqref{eq:Sommerfeld} and \eqref{eq:he} outside some ball. Precisely, we show that the HSM formulations are well-posed and equivalent to the original scattering problems provided we require that each half-plane trace $\varphi^j$  is locally in $L^2$ and has the asymptotic behaviour predicted by \eqref{eq:ffp}, meaning that
\begin{equation} \label{eq:traceRC}
\varphi^j(\bsx) = c_\pm^j \, \frac{\re^{\ri k r}}{r^{1/2}}\left(1+O(r^{-1})\right),
\end{equation}
for some constants $c_\pm^j$, as $\bsx$ tends to infinity in the $\pm$ directions along the infinite boundary of the half-plane. Note that condition \eqref{eq:traceRC} can be seen as a radiation condition for the half-plane trace $\varphi^j$, playing the role in the HSM formulation for $k>0$ that the Sommerfeld radiation condition \eqref{eq:Sommerfeld} plays in the formulation of the original scattering problem.

 Since (i) the HSM system of equations is derived from the unique solution to the scattering problem and (ii) the half-plane traces  of the solution to the scattering problem do satisfy \eqref{eq:traceRC} (because \eqref{eq:ffp} holds), we will see that existence of a solution holds by construction. 
The challenge  remaining in order to show well-posedness is to establish uniqueness. There are two difficult aspects to this challenge. The first is to show that, if $\varphi^j$ satisfies \eqref{eq:traceRC} and $u$ is the solution reconstructed from $\varphi^j$ in the corresponding half-plane, so that $u$ satisfies \eqref{eq:he} in the half-plane and $u=\varphi^j$ on the half-plane boundary, then $u(\bx)= O(r^{-1/2})$ as $r\to\infty$ and $u$ satisfies an appropriate version of the Sommerfeld radiation condition \eqref{eq:Sommerfeld}. These properties are established in \S\ref{sec:half-plane}. The second challenge, given that we assume  {\it ab initio} only that the traces $\varphi^j$ are locally $L^2$, is to show that each trace $\varphi^j$ is locally in the trace space $H^{1/2}$,  so that the reconstructed solution $u$ is locally in $H^1$ as required. But this difficulty arises already in the corresponding formulation in the dissipative case, and the proof in that case is sketched  in \cite[\S3.3]{BB-Fliss-Tonnoir-2018}. As preparation for the more difficult non-dissipative case, we expand on that argument in \S\ref{sec:uniquenessD} below.

\subsection{The significance of our results} The main significance of Theorems \ref{thm:HSMMUnique} and \ref{thm:HSMMUnique_casgeneral} is that these are the first well-posedness results in the important non-dissipative case for the HSM method, a method which, as discussed above, has already proved effective for computing scattering by localised inhomogeneities in a  range of complex backgrounds. Our theorems, challenging as they are to prove, are for the simplest case when the background is homogeneous and isotropic. However, we expect that these results and formulations will be an important stepping stone to more complex cases, and we discuss this further in concluding remarks to the paper.

Our main uniqueness result, Theorem \ref{thm:HSMMUnique}, is also important because it is a crucial ingredient in the proof of well-posedness in \cite[\S5]{FrenchBritish1} of the so-called {\it complex-scaled HSM method}, proposed recently in \cite{FrenchBritish1}. This new formulation is a version of the HSM method in the non-dissipative case that, in the spirit of complex-scaling in PML (e.g., \cite{Col-Mon-1998}), achieves the $L^2$ framework of the dissipative case by analytically continuing the half-plane traces into the complex plane, so that the original half-plane boundary is replaced by a path in the complex plane on which the (analytically-continued) trace is exponentially decreasing.

Additionally, we expect that the results that we establish (Lemma \ref{lem:UjinL2}, \S\ref{sec:half-plane}) on properties of the half-plane solution operator, that takes the trace on the boundary of a half-plane and recovers the half-plane solution, will be of independent interest. Indeed, there is large interest in so-called {\it rough surface scattering problems}, where the Helmholtz equation (or more complicated vector equations) are solved in a non-locally perturbed half-plane $D$ with boundary or transmission conditions on the rough surface $\partial D$ (e.g., \cite{CWZ:98,CWRoZh99,ACWH:02,ArensHohage05,LeRi10,CWElschner10,PinelBoulier13,HuElschner14,Zhang18,HuLuRathsfeld:21}). In this context it is usual, to ensure uniqueness, to impose the so-called {\it upwards propagating radiation condition} (e.g., \cite{CWRoZh99,ArensHohage05,LeRi10,CWElschner10,HuElschner14,Zhang18}), which is precisely a requirement that, in some half-plane above the rough surface, the solution can be represented as the action of this half-plane solution operator on some $L^\infty$ data on the boundary of the half-plane. This rough surface scattering application, in particular the use of the upwards propagating radiation condition, has driven significant study of this half-plane solution operator (e.g.\ \cite{HalfPlaneRep,CW-Impedance-1997}, \cite[\S2]{CWZ:98}, \cite{ArensHohage05}, \cite[\S2.3]{HuElschner14}). The results in \S\ref{sec:half-plane} contribute to this study, shedding light on the relationship between the upwards propagating radiation condition and the Sommerfeld radiation condition, complementing previous work, especially \cite[Theorem 2.9]{CWZ:98}, \cite{ArensHohage05}, and \cite{HuLuRathsfeld:21}.

\subsection{The scattering problem and the structure of the paper} \label{sec:scattering problem} Let us spell out in more detail the 2D scattering problem that we consider in this paper. The propagation domain $\Omega$ is $\R^2$, or $\R^2$ minus a set of bounded Lipschitz obstacles so that $\Omega$ is a Lipschitz domain. We seek $u$ that satisfies an inhomogeneous Helmholtz equation,
\begin{equation}\label{eq:helmholtz-variable}
			-\Delta u -  k^2\,\rho u  = f \; \text{in} \; \Omega.\\[4pt]
\end{equation}
Here $\rho$ is a function in $L^\infty(\Omega)$ 
such that $\rho-1$ is compactly supported, so that, if $\rho$ is  real-valued and bounded below by a positive constant, \eqref{eq:helmholtz-variable} models propagation in a domain with a local perturbation in wave speed\footnote{It is straightforward to incorporate more elaborate local behaviour (e.g.\ local anisotropies, or local inhomogeneities in density as well as wave speed), as long as the resulting problem remains well-posed.}. The source term $f\in L^2(\Omega)$ is also assumed to be compactly supported, so that \eqref{eq:he} holds outside some ball. We suppose that the wavenumber is such that $k>0$ in the non-dissipative case, and such that $\Re(k)>0$ and $\Im(k)>0$ in the dissipative case. In the case when $\Omega\subsetneq \R^2$ we require also a standard (e.g. Dirichlet or Neumann) boundary condition on $\Gamma:= \partial \Omega$.  As is standard we seek a solution $u\in H^1(\Omega)$ in the dissipative case and  $u\in H^1_{\mathrm{loc}}(\Omega):=\{v_{|\Omega}; v\in H^1_{\mathrm{loc}}(\R^2)\}$ in the case $k>0$. In the non-dissipative case we require also that $u$ satisfies the Sommerfeld radiation condition \eqref{eq:Sommerfeld}.

To explain the HSM method, our uniqueness argument, and the role of the radiation condition \eqref{eq:traceRC}, we consider two specific instances of the problem \eqref{eq:helmholtz-variable}. To get the main ideas across and to prove the key uniqueness and well-posedness result, Theorem \ref{thm:HSMMUnique}, we focus first on the case when $\Omega = \R^2\setminus \overline{\O}$, for some convex polygon $\O$, with $\rho\equiv 1$ and $f\equiv 0$, so that \eqref{eq:he} holds in $\Omega$, imposing a Dirichlet boundary condition $u=g$ on $\Gamma$. 
In this case, where $\Omega$ is the exterior of a convex polygon and \eqref{eq:he} holds in $\Omega$, the HSM formulation is a system of second kind integral equations, in which the unknowns are the traces of $u$ on the finite number of half-planes that abut the sides of the polygon $\O$. In \S\ref{section-HSMcomplexfreq} we recall the HSM formulation  for this problem in the dissipative case, in particular how uniqueness is proved. In \S\ref{section-HSMrealfreq} we prove our main Theorem \ref{thm:HSMMUnique}, i.e., we prove well-posedness in the non-dissipative case under the additional radiation condition \eqref{eq:traceRC}.

In \S\ref{sec-CHSMgeneralcase} we consider the more involved  case where $f$ is non zero and/or $\rho\not\equiv 1$, in which case the second kind integral equation formulation is coupled to a local variational formulation in a bounded region  $\Omega_b$ containing $\partial \Omega$ and the supports of $f$ and $\rho-1$. In this case we assume, to be specific, that the boundary condition on $\partial \Omega$ is a homogeneous Neumann condition. (The changes needed, to the formulation and well-posedness argument, to address other boundary conditions, and/or inhomogeneities in the boundary condition, indeed to include other types of compactly supported inhomogeneities, are straightforward.)  The main result in this case, the uniqueness and well-posedness result Theorem \ref{thm:HSMMUnique_casgeneral}, is proved by use of Theorem \ref{thm:HSMMUnique} and by adapting the proof of \cite[Proposition 6.1]{FrenchBritish1}.
Section \ref{sec:conclusion} provides a brief conclusion and gives an indication of more complex configurations to which the same methods and arguments are expected to apply.

\section{The HSM method for complex wavenumber}
\label{section-HSMcomplexfreq}
\subsection{The Half-Space Matching formulation} \label{sec:HSMform}
In this section, as preparation for studying the HSM method for real wavenumber, we first recall what is known about the method in the dissipative case. For this purpose, {as discussed in \S\ref{sec:scattering problem},} we consider the Dirichlet problem for complex wavenumber ($\Im(k)>0$, $\Re(k)>0$) in the exterior $\Omega$ of a convex polygon  $\O$.  Thus, for given data $g\in H^{1/2}(\Gamma)$, where $\Gamma=\partial \Omega =\partial\O$, we consider the Dirichlet problem
\begin{equation}\label{pb:probleme_Dir_diss}
	\begin{cases}
		-\Delta u -  k^2\,u  = 0 \; \text{in} \; \Omega,\\[4pt]
		u=g\;\text{on}\;\Gamma.
	\end{cases}
\end{equation}
It is well known that Problem \eqref{pb:probleme_Dir_diss} has a unique solution $u\in H^1(\Omega)$.


If $\O$ is a polygon with $\NO$ edges  denoted $\Gamma^1,\ldots ,\Gamma^{\NO}$, the domain $\Omega$ is the union of $\NO$ overlapping half-planes $\Omega^j$, $j=1,...,\NO$,  that abut the $\NO$ edges of the polygon $\O$. The angle between $\Gamma^j$ and $\Gamma^{j +1}$ is denoted as $\Theta^{j, j +1}$ or equivalently $\Theta^{j+1,j}$, where, here and in the sequel, $j\in\Z/\NO\Z$ where $\Z/\NO\Z$ is the ring of integers modulo $\NO$. This convenient notation means that $j=0$ is equivalent to $j=\NO$ and $j=-1$  to $j=\NO-1$.
Note finally that, because of the convexity, one has $0 < \Theta^{j, j+1} < \pi$ for all $j\in\Z/\NO\Z$.

\begin{figure}[ht]
	\centering
	\begin{subfigure}{.5\textwidth} 
		\begin{tikzpicture}[scale=0.6]
			\draw[rotate=-30] (-4.5,0) -- (4.5,0);
			\draw[rotate=-30] ({sqrt(3)+1},-3) -- ({1-5/3*sqrt(3)}, 5);
			\draw[rotate=-30] ({-sqrt(3)-1},-3) -- ({5/3*sqrt(3)-1}, 5);
			\filldraw[rotate=-30, fill=gray!40!white, draw=black] (-1, 0) -- (1, 0) -- (0, {sqrt(3)}) -- (-1, 0) ;
			\draw[densely dotted,rotate=-30, ->] (-0.1, -0.4) -- (-0.1, -0.9);
			\draw[densely dotted,rotate=-30, ->] (-0.1, -0.4) -- (0.4, -0.4);
			\draw[densely dotted,->](1.3,0.5) -- (1.8,0.5);
			\draw[densely dotted,->](1.3,0.5) -- (1.3,1);
			\draw[densely dotted,rotate=30, ->] (0.7, 1.3) -- (0.7, 1.8);
			\draw[densely dotted,rotate=30, ->] (0.7, 1.3) -- (0.2, 1.3);
			\node at (0.3, 0.5) {$\mathcal{O}$};
			\node at (1.3, -3.5) {$\Sigma^1$};
			\node at (3.5, 3.5) {$\Sigma^2$};
			\node at (-3.5, 1.5) {$\Sigma^3$};
			\node at (2., 0.3) {$x^1_1$};
			\node at (1.55, 1.2) {$x^1_2$};
			\node at (-0.3, 2.2) {$x^2_1$};
			\node at (-0.8, 1.2) {$x^2_2$};
			\node at (-0.7, -1.1) {$x^3_1$};
			\node at (0.4, -0.7) {$x^3_2$};
			\node at (1.8, 2.7) {$\Theta^{1,2}$};
			\node at (-2.5, 0.5) {$\Theta^{2,3}$};
			\node at (1.6, -1.7) {$\Theta^{3,1}$};
			\draw[densely dotted] (-1.8,1) arc (150:210:1);
			\draw[densely dotted](0.9,-1.5) arc (270:330:1);
			\draw[densely dotted] (1.7, 2) arc (30:90:1);
		\end{tikzpicture}
	\end{subfigure}%
	\begin{subfigure}{.5\textwidth} 
		\begin{tikzpicture}[scale=0.5]
			\draw (-5.5,-1) -- (5.5,-1);
			\draw (-5.5,1) -- (5.5,1);
			\draw (-2,-4.5) -- (-2, 4.5);
			\draw (2,-4.5) -- (2, 4.5);
			\filldraw[fill=gray!40!white, draw=black] (-2, -1) -- (2, -1) -- (2,1) -- (-2, 1) -- (-2, -1) ;
			\draw[densely dotted,->] (2.4, -0.3) -- (2.9, -0.3);
			\draw[densely dotted,->] (2.4, -0.3) -- (2.4, 0.2);
			\draw[densely dotted,->] (-2.4, 0.3) -- (-2.9, 0.3);
			\draw[densely dotted,->] (-2.4, 0.3) -- (-2.4, -0.2);
			\draw[densely dotted,->] (0.3, 1.4) -- (0.3, 1.9);
			\draw[densely dotted,->] (0.3, 1.4) -- (-0.2, 1.4);
			\draw[densely dotted,->] (-0.2, -1.4) -- (-0.2, -1.9);
			\draw[densely dotted,->] (-0.2, -1.4) -- (0.3, -1.4);
			\node at (0, 0) {$\mathcal{O}$};
			\node at (2.5, -4.2) {$\Sigma^1$};
			\node at (5., 1.5) {$\Sigma^2$};
			\node at (-2.5, 4.2) {$\Sigma^3$};
			\node at (-5., -1.5) {$\Sigma^4$};
			\node at (3.4, -0.3) {$x^1_1$};
			\node at (2.7, 0.5) {$x^1_2$};
			\node at (0.5, 2.3) {$x^2_1$};
			\node at (-0.6, 1.4) {$x^2_2$};
			\node at (-3.3, 0.5) {$x^3_1$};
			\node at (-2.8, -0.5) {$x^3_2$};
			\node at (-0.2, -2.2) {$x^4_1$};
			\node at (0.7, -1.5) {$x^4_2$};
			\node at (3.3, 2.) {$\Theta^{1,2}$};
			\node at (-3.3, 2.) {$\Theta^{2,3}$};
			\node at (-3.3, -2.) {$\Theta^{3,4}$};
			\node at (3.3, -2.) {$\Theta^{4,1}$};
			\draw[densely dotted] (3,1) arc (0:90:1);
			\draw[densely dotted] (-2,2) arc (90:180:1);
			\draw[densely dotted] (-3,-1) arc (180:270:1);
			\draw[densely dotted] (2,-2) arc (270:360:1);
		\end{tikzpicture}
	\end{subfigure}
	\caption{Examples of polygons $\mathcal{O}$ for $\NO = 3$ and $4$ and associated notations.}
	\label{fig:rmd}
\end{figure}
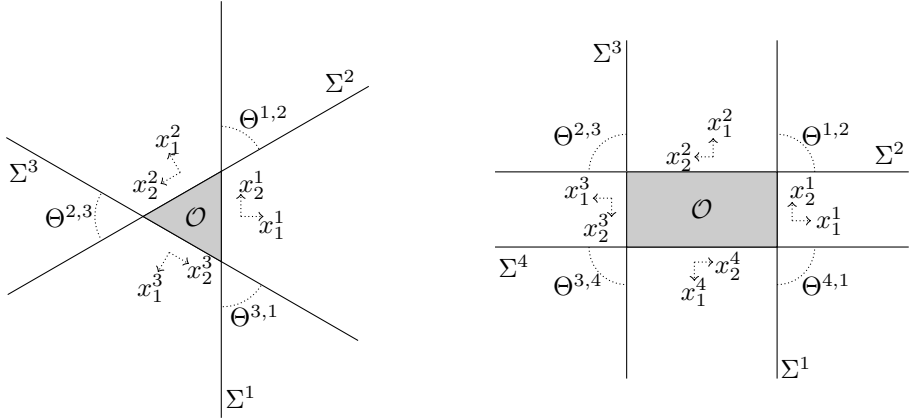

It is convenient to make use of  local coordinate systems $\bsx^j=(x^j_1,x^j_2)$ in each half-space $\Omega^j$. The origin of all of them is the centroid $O$ of the polygon $\mathcal{O}$. We define the Cartesian coordinate system $(O, x^1_1, x^1_2)$ such that the axis $Ox^1_1$ is orthogonal to $\Gamma^1$ and directed into the exterior of the polygon, while the axis $Ox^1_2$ is $\pi/2$ counter clockwise to $Ox^1_1$. The other local coordinate systems $(O, x^j_1, x^j_2)$ are defined recursively as follows:
\begin{equation}
	\forall  j\in \Z/\NO\Z,  \quad \begin{array}{|l}
		x^{j+1}_1  :=  \cos ( \Theta^{j, j +1})\, x^j_1 + \sin (\Theta^{j, j +1})\, x^j_2, \\  [5pt]
		x^{j+1}_2   := -\sin (\Theta^{j, j +1})\,x^j_1 +\cos (\Theta^{j, j +1})\, x^j_2. \end{array}
\end{equation}
Defining $a^j$, for $j=1,...,\NO$, to be the distance of the centroid of the polygon to the edge $\Gamma^j$, each half-plane $\Omega^j$ is defined in the local coordinate system $(O, x^j_1, x^j_2)$
as
\[
\Omega^j := \{ x^j_1 >a^j\} \times \{ x^j_2 \in \R\},\]
and its boundary, denoted by $\Sigma^j$, is given by
\[
\Sigma^j := \{ x^j_1 = a^j\} \times \{ x^j_2 \in \R\}.\]
As explained in the introduction, the formulation uses the representation of the solution in each half-plane $\Omega^j$ in terms of its trace on $\Sigma^j$. More precisely, let us denote
\begin{equation}
	\label{eq:def_traces}
	\varphi^j:=u\big|_{\Sigma^j}\quad \text{for}\;j =1,...,\NO,
\end{equation}
so that
\begin{equation}
	\label{eq:def_traces2}
	u\big|_{\Omega^j}=U^j(\varphi^j)\quad \text{for}\;j = 1,...,\NO,
\end{equation}
where, for $j=1,...,\NO$ and $\psi\in H^{1/2}(\Sigma^j)$,  $U^j(\psi)\in H^1(\Omega^j)$ is the unique solution of
\begin{equation}\label{eq:half-space}
	\begin{array}{|lcr}
		-\Delta U^j -  k^2\, U^j = 0 & \text{in} & \Omega^j,\\[4pt]
		\;U^j=\psi&\text{on}&\Sigma^j,
	\end{array}
\end{equation}
this solution being well-defined since $\Im(k)\neq 0$.
We can express $U^j(\psi)$ explicitly in terms of its trace $\psi$  using a Green's function  representation:
\begin{equation} \label{eq:hprGreen}
	U^j (\psi) ({\bsx}^j) = \int_{\Sigma^j}\frac{\partial G^j({\bsx}^j,{\bsy}^j)}{\partial n({\bsy}^j)}\,\psi({\bsy}^j)\, \rd s({\bsy}^j),\quad {\bsx}^j \in \Omega^j,
\end{equation}
where $G^j({\bsx}^j,{\bsy}^j)$ is the Dirichlet Green's function for $\Omega^j$ and $n({\bsy}^j)$ is the unit normal to $\Sigma^j$ that points into $\Omega^j$. Explicitly, $G^j({\bsx}^j,{\bsy}^j)= \Phi({\bsx}^j,{\bsy}^j)-\Phi(\widetilde{\bsx}^j,{\bsy}^j)$, with $\widetilde{\bsx}^j$ the image of ${\bsx}^j$ in $\Sigma^j$, where
$\Phi({\bsx},{\bsy})$ is the standard fundamental solution of the Helmholtz equation defined by
\begin{equation} \label{eq:Greenfct}
	\Phi({\bsx},{\bsy}) := \frac{\ri}{4}H_0^{(1)}( k|\bsx-\bsy|), \quad {\bsx,\,\bsy\in \R^2, \;\;\bsx\neq \bsy},
\end{equation}
where $H^{(1)}_n$ is the Hankel function of the first kind of order $n$.
This leads finally to
\begin{equation} \label{eq:hprPhi}
	U^j (\psi) ({\bsx}^j) =2 \int_{\Sigma^j}\frac{\partial \Phi({\bsx}^j,{\bsy}^j)}{\partial y_1^j}\,\psi({\bsy}^j)\, \rd s({\bsy}^j),\quad {\bsx}^j \in \Omega^j,
\end{equation}
which  can be rewritten as
\begin{equation} \label{eq:hpr}
	U^j (\psi) ({\bsx}^j) = \int_{\R}\mathcal{H}(k;x_1^j-a^j,x_2^j-y_2^j)\,\psi(a^j,y_2^j)\, \rd y_2^j,\quad {\bsx}^j=(x_1^j,x_2^j) \in \Omega^j,
\end{equation}
where we have set
\begin{equation} \label{eq:hpr_kernel}
\mathcal{H}(k;x_1,x_2):=\frac{\ri kx_1}{2}\frac{H^{(1)}_1( k |\bsx|)}{|\bsx|}.
\end{equation}
To derive a system of equations whose unknowns are the traces $\varphi^j$ of the solution, it suffices to observe that the half-plane representations must coincide where they coexist.
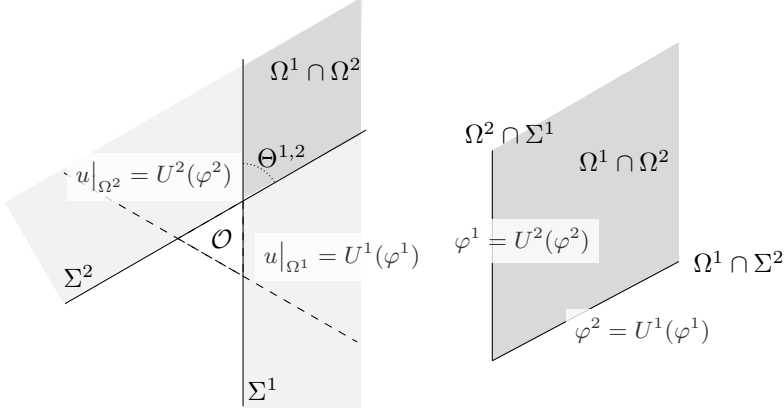
\begin{figure}[ht]
	\centering
	\begin{subfigure}{.450\textwidth} 
		\begin{tikzpicture}[scale=0.5]
			\fill[color=gray!10,rotate=-60] (-4,-4) -- (-4,7) -- (-0.9,5) -- (-0.9,-4);
			\fill[color=gray!10] (0.9,-4)rectangle(4,5);
			\fill[color=gray!30] (0.9,1.6) -- (0.9,5.1) -- (4,6.9) -- (4,3.25);
			\draw[rotate=-30,dashed] (-4.5,0) -- (4.5,0);
			\draw[rotate=-30] ({sqrt(3)+1},-3) -- ({1-5/3*sqrt(3)}, 5);
			\draw[rotate=-30] ({-sqrt(3)-1},-3) -- ({5/3*sqrt(3)-1}, 5);
			\draw[rotate=-30, dashed] (-1, 0) -- (1, 0) -- (0, {sqrt(3)}) -- (-1, 0) ;
			\node at (0.3, 0.5) {$\mathcal{O}$};
			\node at (1.4, -3.5) {$\Sigma^1$};
			\node at (-3.5, -0.5) {$\Sigma^2$};
			\node at (1.9, 2.7) {$\Theta^{1,2}$};
			\draw[densely dotted] (1.7, 2) arc (30:90:1);		
			\draw (3.5,0) node[fill=white, opacity=.8]   {\small $u\big|_{\Omega^1}=U^1(\varphi^1)$};
			\draw (-1.5,1.5) node[above,fill=white, opacity=.8]  {\small $u\big|_{\Omega^2}=U^2(\varphi^2)$};
			 \draw (2.8,5) node  {${\Omega^1}\cap{\Omega^2}$};
		\end{tikzpicture}
	\end{subfigure}
	\begin{subfigure}{.450\textwidth} 
		\begin{tikzpicture}[scale=0.8]
			\fill[color=gray!30] (-3,0) -- (-3,3.5) -- (0.1,5.3) -- (0.1,1.65);
			\draw (-3,0) -- (-3,3.5);
			\draw (-3,0) -- (0.1,1.65);
			\node at (-2.7, 3.8) {$\Omega^2\cap\Sigma^1$};
			\node at (1.1, 1.65) {$\Omega^1\cap \Sigma^2$};
			\draw (-0.5,0.5) node[fill=white, opacity=.8]  {\small $\varphi^2=U^1(\varphi^1)$};
			\draw (-2.5,2) node[fill=white, opacity=.8]  {\small $\varphi^1=U^2(\varphi^2)$};
			\draw (-.8,3.) node[above]  {${\Omega^1}\cap{\Omega^2}$};
		\end{tikzpicture}
	\end{subfigure}  %
	%
	%
	\caption{Compatibility condition in $\Omega^1\cap\Omega^2$.}
	\label{fig:comp_cond}
\end{figure}
For instance, in the overlapping zone $\Omega^1\cap\Omega^2$ (see Figure \ref{fig:comp_cond}) we have
\begin{equation}\label{eq:comp_quarter}
	u=U^1(\varphi^1)=U^2(\varphi^2)\quad \text{in}\;\Omega^1\cap\Omega^2,
\end{equation}
and in particular
\begin{equation}\label{eq:comp1}
	\varphi^2=U^1(\varphi^1)\quad \text{on}\;\Omega^1\cap\Sigma^2,
\end{equation}
which is a relation linking $\varphi^1$ and $\varphi^2$. By introducing the operator $D_{1,2}:L^2(\Sigma^1)\rightarrow L^2(\Omega^1\cap\Sigma^2)$ defined by
\[
 D_{1,2}\,\psi:=U^1(\psi)\big|_{\Omega^1\cap\Sigma^2},\quad \psi \in L^2(\Sigma^1),
\]
the relation \eqref{eq:comp1} can be rewritten as
\begin{equation}\label{eq:comp1_bis}
	\varphi^2=D_{1,2}\varphi^1\quad \text{on}\;\Omega^1\cap\Sigma^2.
\end{equation}
From \eqref{eq:comp_quarter}, we deduce similarly
\[
\varphi^1=U^2(\varphi^2)\quad \text{on}\;\Omega^2\cap\Sigma^1,
\]
another relation linking $\varphi^1$ and $\varphi^2$. Repeating this with $1$ and $2$ replaced, respectively, by $j$ and $j+1$, we get $2\mathcal{N}$ equations linking the $\mathcal{N}$ traces. Let us introduce, for all $j\in \Z/\mathcal{N}\Z$, the operators $D_{j,j\pm 1}:H^{1/2}(\Sigma^j)\rightarrow H^{1/2}(\Omega^j\cap\Sigma^{j\pm 1})$ defined by
\begin{equation}\label{eq:DtD_jjplus1}
D_{j,j\pm1}\,\psi:=U^j(\psi)\big|_{\Omega^j\cap\Sigma^{j\pm1}},\quad\psi \in H^{1/2}(\Sigma^j).
\end{equation}
Then the compatibility relations between all the traces can be written as:
\begin{equation}\label{eq:syst_comp}
	\forall j\in \Z/\mathcal{N}\Z,\quad\begin{array}{|lcl}
		\varphi^j=D_{j-1,j}\,\varphi^{j-1},\quad \text{on}\;\Omega^{j-1}\cap\Sigma^j,\\
		\varphi^j=D_{j+1,j}\,\varphi^{j+1},\quad \text{on}\;\Omega^{j+1}\cap\Sigma^j.
	\end{array}
\end{equation}
This system of equations has to be completed with the Dirichlet boundary condition, rewritten as
\begin{equation}
	\label{eq:syst_comp_BC}
	\varphi^j=g|_{\Gamma^j}\quad \text{on}\;\Gamma^j,\quad j=1,...,\NO.
\end{equation}
Note that for all $j$ the operators $D_{j,j\pm1}$ can be given explicitly by using the representation \eqref{eq:hprPhi} for $U^j$. They are integral operators and more precisely double-layer potential operators (in the sense, e.g., of \cite{CoKr:98} or \cite{Ch:84}). 

One can easily  check that \eqref{eq:syst_comp}-\eqref{eq:syst_comp_BC} is equivalent to the original problem \eqref{pb:probleme_Dir_diss} and so is well-posed. We recall here the proof of this result because its general idea will be used in the results  proved later in the paper.
\begin{theorem}
\label{the:equivalenceH1/2}	
 In the case where $\Re(k)>0$ and  $\Im(k)>0$, there exists, for each $g\in H^{1/2}(\Gamma)$, a unique solution $(\varphi^1,..., \varphi^\NO)\in H^{1/2}(\Sigma^1)\times\ldots\times H^{1/2}(\Sigma^{\NO})$  of
\eqref{eq:syst_comp}-\eqref{eq:syst_comp_BC}.
\end{theorem}
Existence was proven by construction of \eqref{eq:syst_comp}-\eqref{eq:syst_comp_BC} from the solution $u$ of \eqref{pb:probleme_Dir_diss}, with $\varphi^j:=u|_{\Sigma^j}$ for $j=1,...,\NO$. Uniqueness follows from the following proposition since \eqref{pb:probleme_Dir_diss} is well-posed.
\begin{proposition}
\label{prop:equivalenceH1/2}	
Suppose $\Re(k)>0$ and $\Im(k)>0$ and let $(\varphi^1,..., \varphi^\NO)$ $\in$ $H^{1/2}(\Sigma^1)$ $\times\ldots\times H^{1/2}(\Sigma^{\NO})$ be a solution to \eqref{eq:syst_comp}-\eqref{eq:syst_comp_BC} (with $g\in H^{1/2}(\Gamma)$). Then, for all $j\in \Z/\mathcal{N}\Z$,
\begin{equation}
	\label{eq:compatibility}
U^j(\varphi^j)=U^{j+1}(\varphi^{j+1})\quad \text{in}\; \Omega^j\cap\Omega^{j+1},
\end{equation}
and the function  defined by
$	u=U^j(\varphi^j)$ in $\Omega^j$, $j=1,...,\NO$,
is the unique solution $u\in H^1(\Omega)$ of \eqref{pb:probleme_Dir_diss}.
\end{proposition}
\begin{proof}
To prove \eqref{eq:compatibility}, let us set $v=U^j(\varphi^j)-U^{j+1}(\varphi^{j+1})$. By definition of the half-plane representations $U^j$ and $U^{j+1}$, it is clear that $v\in H^1(\Omega^j\cap\Omega^{j+1})$ and that $v$ satisfies $\Delta v+k^2v=0$ in $\Omega^j\cap\Omega^{j+1}$. Moreover, the compatibility conditions \eqref{eq:syst_comp} imply that  $v$ vanishes on the boundary of $\Omega^j\cap\Omega^{j+1}$. Using that $\Im(k)\neq 0$, one deduces that $v=0$ by uniqueness of the solution of the Dirichlet problem in $\Omega^j\cap\Omega^{j+1}$. The rest of the proof is straightforward.
\end{proof}
It has been shown in previous papers \cite{BB-Fliss-Tonnoir-2018,Bon-Fli-Tja-2019} that, for both mathematical analysis and computation, it is more convenient to consider the Half-Space Matching formulation in an $L^2$-framework, which means that $\varphi^j$ is sought in $L^2(\Sigma^j)$ instead of $H^{1/2}(\Sigma^j)$ so that the formulation is
\begin{equation}\label{eq:HSMM_complexe}
	\begin{array}{c}
		\text{Find}\; (\varphi^1,...,\varphi^\NO) \in L^2(\Sigma^1)\times\ldots\times L^2(\Sigma^{\NO})\;\text{such that for }j\in \Z/\mathcal{N}\Z\\[5pt]
		\begin{array}{|ll}
			\varphi^j=D_{j-1,j}\,\varphi^{j-1},& \text{on}\;\Sigma^j\cap\Omega^{j-1},\\
			\varphi^j=g|_{\Gamma^j},& \text{on}\;\Gamma^j,\\
			\varphi^j=D_{j+1,j}\,\varphi^{j+1},& \text{on}\;\Sigma^j\cap\Omega^{j+1}.
		\end{array}
	\end{array}
\end{equation}
Let us emphasize that this makes sense since $D_{j,j\pm 1}\psi$ is well-defined by \eqref{eq:DtD_jjplus1} and \eqref{eq:hpr} for all $\psi\in L^2(\Sigma^j)$ and the operators $D_{j,j\pm 1}$ are in fact continuous from $L^2(\Sigma^j)$ to $L^2(\Omega^j\cap\Sigma^{j\pm 1})$ \cite{BB-Fliss-Tonnoir-2018}.
It has been shown in \cite{BB-Fliss-Tonnoir-2018} that problem  \eqref{eq:HSMM_complexe} is of Fredholm type. More precisely, by rewriting it in an operator form, it is proven that the associated operator is Fredholm of index 0. This means that existence of a solution is equivalent to its uniqueness. We refer the reader to \cite{BB-Fliss-Tonnoir-2018} for more detail. Here, we focus on the question of uniqueness. We need to prove that, for a solution of \eqref{eq:HSMM_complexe}, $g=0$ implies $\varphi^j=0$ for all $j=1,...,\NO$. The sketch of the proof of this uniqueness result, which is much less straightforward than in the $H^{1/2}$ framework, has been given in \cite{BB-Fliss-Tonnoir-2018}. We give a detailed presentation of this proof in the following paragraph, in a form that will be directly used for our main result, which is uniqueness for the case of a real wavenumber $k$.

\subsection{The uniqueness result in an $L^2$-framework} \label{sec:uniquenessD}

The difficulty in proving a uniqueness result for problem \eqref{eq:HSMM_complexe} comes from the fact that the half-plane solutions $U^j(\varphi^j)$ (with  $U^j(\varphi^j)$ defined now by \eqref{eq:hprGreen}, \eqref{eq:hprPhi}, or \eqref{eq:hpr}) are in general not in $H^1(\Omega^j)$, assuming only that $\varphi^j\in L^2(\Sigma^j)$.
But we will take advantage of the following lemma which summarizes key properties of the half-plane solutions with $L^2$ and/or { $L^\infty$ Dirichlet boundary data.}
\begin{lemma}
	\label{lem:UjinL2}
	Suppose that $\Im(k)\geq 0,\Re(k)>0$. Let $\phi\in L^2(\R)+ L^\infty(\R)$
	and define $u(\bsx)$, for $\bsx=(x_1,x_2)\in \R^2_+:=(0,\infty)\times \R$, by
	\begin{equation} \label{eq:hplem-1}
		u(\bsx) := \int_\R \mathcal{H}(k;x_1,x_2-t)\phi(t)\, \rd t,
	\end{equation}
where $\mathcal{H}$ is defined by \eqref{eq:hpr_kernel}.
Then  the integral is well-defined as a Lebesgue integral for $\bsx\in\R^2_+$,  $u\in {C}^\infty(\R^2_+)$, and $\Delta u + k^2\, u = 0 $ in $\R^2_+$. Moreover, if $\phi$ is continuous on an open interval $I\subset \R$, then $u\in C(\R^2_+\cup \gamma)$ where $\gamma:=\{(0,t);\;t\in I\}$, and $u(0,t)=\phi(t),\;t\in I$. Finally, if $\phi\in L^2(\R)$ and  $\Im(k)> 0$, then
	$u\in L^2(\R^2_+)$ and there exists a constant $C_\infty>0$ independent of $\phi$ such that
	\begin{equation} \label{eq:uboundL2}
		\|u\|_{L^2(\R^2_+)} \leq C_\infty\|\phi\|_{L^2(\R)},
	\end{equation}
	whereas if $\Im(k)= 0$ then, for any $L>0$, $u\in L^2(\omega_L)$ where $\omega_L:=\{\bsx\in\R^2_+;\;x_1<L\}$, and there exists a constant $C_L>0$ independent of $\phi$ such that
	\begin{equation} \label{eq:uboundL2_rek}
		\|u\|_{L^2(\omega_L)} \leq C_L \|\phi\|_{L^2(\R)}.
	\end{equation}
\end{lemma}
\begin{proof}
	From asymptotics of the Hankel function $H^{(1)}_1$ for large and small argument, it follows that, for every $k$ with $\Im(k)\geq 0$, $\Re(k)> 0$, there exists a $C>0$ such that
	\begin{equation}
	\label{eq:Hbound}
	|\mathcal{H}(k;x_1,x_2)|\leq C\frac{x_1}{|\bsx|^2}(1+|\bsx|^{1/2})e^{- \Im(k)|\bsx|}.	
	\end{equation}
	It follows that, for $\bsx\in\R^2_+$, the integral in \eqref{eq:hplem-1} is well-defined as a Lebesgue integral for any $\phi\in L^2(\R)+ L^\infty(\R)$.
To see that $u\in {C}^\infty(\R^2_+)$ and satisfies the Helmholtz equation, it suffices to argue as in \cite[Theorem 3.2]{CW-Impedance-1997}. If $\phi$ is continuous on an open interval $I\subset \R$, to see that $u\in C(\R^2_+\cup \gamma)$ and that $u(0,t)=\phi(t)$ for $t\in I$, it is enough to show that $u\in C(\R^2_+\cup \tilde \gamma)$  and that $u(0,t)=\phi(t)$ for $t\in \tilde{I}$ for every compact $\tilde\gamma:=\{(0,t);\;t\in \tilde{I}\}\subset\gamma$. So suppose $\tilde{I}\subset I$ is compact and write $\phi=\phi_1+\phi_2$ where the support of $\phi_2$ does not intersect  $\tilde{I}$ and $\phi_1$ is bounded and continuous. Correspondingly, we can split $u$ as $u=u_1+u_2$, where $u_j$ is defined by \eqref{eq:hplem-1} with $\phi$ replaced by $\phi_j$. It follows from the definition that $u_2\in {C}(\R^2_+\cup\tilde\gamma)$  with $u_2(0,t)=0$ for $t\in \tilde{I}$, while $u_1\in {C}(\overline{\R^2_+})$ and $u_1(0,t)=\phi_1(t)$ for $t\in \tilde{I}$ by \cite[Theorem 3.1]{HalfPlaneRep}, so that $u(0,t)=\phi(t)$ for $t\in \tilde{I}$.
Concerning the $L^2$-estimates \eqref{eq:uboundL2} and \eqref{eq:uboundL2_rek} let us recall that for $\psi_1\in L^1(\R)$ and $\psi_2\in L^2(\R)$ the convolution product $\psi_1*\psi_2$ belongs to $L^2(\R)$, with the estimate (Young's convolution inequality)
$$\|\psi_1*\psi_2\|_{L^2(\R)}\leq \|\psi_1\|_{L^1(\R)}\|\psi_2\|_{L^2(\R)}.$$
For a given $x_1>0$, this implies that
$$\int_{\R}|u(x_1,x_2)|^2dx_2\leq \left[\int_{\R}|\mathcal{H}(k;x_1,x_2)|dx_2\right]^2\|\phi\|_{L^2(\R)}^2.$$
From \eqref{eq:Hbound}, we deduce that, for some $C'>0$,
$$\int_{\R}|\mathcal{H}(k;x_1,x_2)|dx_2\leq C'(1+x_1)^{1/2}e^{- \Im(k)x_1},   \quad  x_1>0.$$
One  easily obtains the estimates \eqref{eq:uboundL2} for $\Im(k)> 0$ and \eqref{eq:uboundL2_rek} for $\Im(k)= 0$.
\end{proof}

Now to prove the uniqueness result, we need some results concerning $L^2$ solutions of the homogeneous Dirichlet problem in a domain with corners. We gather in the next lemma the results that we need which are proved in 	\cite[Chapter 2]{Grisvard:92} (see Theorem 2.3.3 and the proof of Theorem 2.3.7).
\begin{lemma}
	\label{lem-singularities}
	Let $Q$ be a bounded polygonal domain with $N$ vertices denoted by $S_1,\ldots, S_N$. Let us suppose that $Q$ has $M<N$ reentrant vertices that can be supposed to be $S_1,\ldots, S_M$ without loss of generality. Let $P:=\{ S_1,\ldots, S_N\}$ and suppose that $w\in  {C}(\bar{Q}\setminus P)\cap {C}^2(Q)$ satisfies $\Delta w=0$ in $Q$ and $w=0$ on $\partial Q\setminus P$. Then, if $w\in L^2(Q)$, there exists $\tilde{w}\in H^1(Q)$ and $M$ complex constants $c_1,\ldots,c_M$ such that
	\[
	w(\bsx)=\tilde{w}(\bsx)+\sum_{m=1}^M c_m\,r_m^{-\pi/\alpha_m}\sin(\pi\theta_m/\alpha_m), \quad \bsx \in Q,	\]
	where, for each $m$, $\alpha_m$ is the interior angle of $Q$ at $S_m$ and $(r_m,\theta_m)$ are the polar coordinates of $\bsx$ whose origin is $S_m$ such that $\theta_m$ is the angle from one of the sides of $\partial Q$ containing $S_m$.
	In particular if $M=0$ ($Q$ is convex) then $w=0$.
\end{lemma}
The previous lemma will be used to prove that particular $L^2$ solutions of the Helmholtz equations in unbounded domains with Dirichlet boundary conditions are in fact $H^1$ in every bounded subset of the domain. To prove that they are $H^1$ in the whole domain, we will resort to the following result, whose proof follows the arguments of \cite[Lemma 2.2]{Spence2014}.
\begin{lemma}
	\label{lem-L2givesH1}
Let $\Pi\subset \R^2$ be an unbounded  open set such that $\Pi_R:=\{\bsx\in\Pi; |\bsx|<R+1\}$ is a Lipschitz domain for all sufficiently large $R>0$, and $v\in L^2(\Pi)$ a function such that $v_{|B}\in H^1(B)$ for all bounded open subsets $B$ of $\Pi$. If $v$ satisfies the Helmholtz equation $\Delta v+k^2v=0$ in $ \Pi$ and its trace is zero on $\partial\Pi$, then  $v\in H^1(\Pi)$.
\end{lemma}
\begin{proof}
 Let us define the truncation function $\chi_R(\bsx):=F_R(|\bsx|)$ where $F_R\in C^\infty(\R^+)$ is chosen such that  $0\leq F_R\leq 1$, $F_R=1$ on $[0,R]$, $F_R=0$ on $[R+1,+\infty)$ and there exists a constant $c>0$ independent of $R$ such that $|F_R'|\leq c\sqrt{F_R}$ (this can be achieved by defining $F_R$ so that $F_R>0$ on $(R,R+1)$ and vanishes quadratically at $R+1$). One has for all $R>0$
	\[
	\int_{\Pi}(\Delta v+k^2v)\overline{v}\chi_R\,d\bsx=0.
	\]
	Since $v\in H^1(\Pi_R;\Delta):= \left\{u\in H^1(\Pi_R);\Delta u\in L^2(\Pi_R)\right\}$
we can apply Green's first identity in $\Pi_R$. Using that $v$ vanishes on $\partial \Pi$, this yields
	\[
	\int_{\Pi}|\nabla v|^2\chi_R\,d\bsx\leq k^2\|v\|^2_{L^2(\Pi)}+\left|\int_{\Pi}\overline{v}\nabla v\cdot\nabla\chi_R\,d\bsx\right|.
	\]
	Since $|\nabla\chi_R|\leq c\sqrt{\chi_R}$, applying { the} Cauchy-Schwarz inequality gives
	\[
	\int_{\Pi}|\nabla v|^2\chi_R\,d\bsx\leq k^2\|v\|^2_{L^2(\Pi)}+c\|v\|_{L^2(\Pi)}\sqrt{\int_{\Pi}|\nabla v|^2\chi_R\,d\bsx},
	\]
	from which we deduce the existence of a constant $C>0$ independent of $R$ such that
	\[
	\int_{\Pi}|\nabla v|^2\chi_R\,d\bsx\leq C\|v\|^2_{L^2(\Pi)}.
	\]
This being true for all $R$, we can conclude  that $v\in H^1(\Pi)$.
\end{proof}
We are now able to prove the following well-posedness result.
\begin{theorem}
	\label{thm:uniqueness-kcomplex}	
If $\Re(k)>0$ and $\Im(k)>0$ then there exists, for each $g\in H^{1/2}(\Gamma)$, a unique solution $(\varphi^1,..., \varphi^\NO)\in L^2(\Sigma^1)\times\ldots\times L^2(\Sigma^{\NO})$  of
\eqref{eq:HSMM_complexe}.
\end{theorem}
This is a direct consequence of the following proposition and Theorem \ref{the:equivalenceH1/2}.
\begin{proposition}\label{prop:trraceL2}
	Suppose $\Re(k)>0$ and $\Im(k)>0$ and let $(\varphi^1,..., \varphi^\NO)$ $\in$ $L^2(\Sigma^1)$ $\times\ldots\times L^2(\Sigma^{\NO})$ be a solution of \eqref{eq:HSMM_complexe} (with $g\in H^{1/2}(\Gamma)$). Then $\varphi^j\in H^{1/2}(\Sigma^j)$, $j=1,...,\NO$.
\end{proposition}
\begin{proof} The proof consists of two steps.
	\begin{enumerate}
		\item Prove the relations \eqref{eq:compatibility} as in Proposition \ref{prop:equivalenceH1/2}	but noting that here the functions $(\varphi^1,..., \varphi^\NO)$ that solve \eqref{eq:HSMM_complexe} are a priori supposed to be only in $L^2$.
		\item  By noting that the traces on $\Sigma^1,\ldots,\Sigma^\NO$ of the unique solution of \eqref{pb:probleme_Dir_diss} form a solution of \eqref{eq:HSMM_complexe} (and are obviously in $H^{1/2}$), it suffices to prove that, if $g=0$, $\varphi^j=0$ for $j=1,...,\NO$. To prove this uniqueness result, we show that, in the case that $(\varphi^1,..., \varphi^\NO)$ satisfies \eqref{eq:HSMM_complexe} with $g=0$, the function  defined (thanks to step 1) by
$	u=U^j(\varphi^j)$ in $\Omega^j$ ($U^j(\varphi^j)$ defined by \eqref{eq:hprPhi})  is in $H^1(\Omega)$, with zero trace. Therefore it is equal to 0 everywhere, as the unique solution of \eqref{pb:probleme_Dir_diss} for $g=0$. We deduce finally that each $\varphi^j$, which we show is the trace of $u$ on $\Sigma^j$, is zero.
	\end{enumerate}

\vspace{2mm}

Step 1.	As in the proof of Proposition \ref{prop:equivalenceH1/2}, we introduce $v=U^j(\varphi^j)-U^{j+1}(\varphi^{j+1})$, where the $\varphi^j\in L^2(\Sigma^j)$ satisfy \eqref{eq:HSMM_complexe}, in particular satisfy \eqref{eq:syst_comp}. A priori, we just know, thanks to Lemma \ref{lem:UjinL2}, that $v\in L^2(\Omega^j\cap\Omega^{j+1})\cap C^\infty(\Omega^j\cap\Omega^{j+1})$ and that $v$ satisfies $\Delta v+k^2v=0$ in $\Omega^j\cap\Omega^{j+1}$. Note that it follows from  \eqref{eq:syst_comp} and Lemma \ref{lem:UjinL2} that the $\varphi^j$'s are in fact continuous on $\Sigma^j\setminus \overline{\Gamma^j}$, so that, from Lemma \ref{lem:UjinL2}, $U^j(\varphi^j)\in C(\overline{\Omega^j}\setminus \overline{\Gamma^j})$ and $U^j(\varphi^j)=\varphi^j$ on ${\Sigma^j\setminus\overline{\Gamma^j}}$. As a consequence, $v$ is continuous in $\overline{\Omega^j\cap \Omega^{j+1}}$, except maybe at $S_{j}$, the intersection of $\Sigma^j$ and  $\Sigma^{j+1}$, and, thanks to \eqref{eq:syst_comp}, $v$ vanishes on $\partial(\Omega^j\cap \Omega^{j+1})\setminus \{S_{j}\}$. It follows by standard reflection and elliptic regularity arguments that $v\in C^\infty \left(\overline{\Omega^j\cap \Omega^{j+1}}\setminus \{S_{j}\}\right)$.
	
 To conclude that $v=0$, we only need to prove that $v\in H^1(\Omega^j\cap\Omega^{j+1})$. Firstly, we show that $v$ is $H^1$ in all bounded subdomains of $\Omega^j\cap\Omega^{j+1}$. Let us introduce a bounded convex polygon $\mathcal{O}'$ such that $\overline{\mathcal{O}}\subset \mathcal{O}'$ (see Figure \ref{fig:illustration-theorem}). Then $Q:=\mathcal{O}'\cap \Omega^j\cap\Omega^{j+1}$ is a convex polygon. Let us   consider the following Dirichlet problem: find $\tilde{v}\in H^1(Q)$ such that
	$$\begin{array}{|ll}
		-\Delta \tilde{v}=k^2v&\mbox{in }Q,\\
		\tilde{v}=v&\mbox{on }\partial Q.
	\end{array}$$
	This problem has  a unique solution as $v\in L^2(Q)$ and the restriction of $v$ to $\partial Q$ is in $H^{1/2}(\partial Q)\cap C(\partial Q)$, since $v\in C^\infty(\overline Q\setminus \{S_j\})$ and,  if we set $v=0$ at $S_j$,  $v|_{\partial Q}$ vanishes in a neighbourhood of $S_j$.  Further, by standard elliptic regularity arguments,  $\tilde{v}\in C^2(Q)$ and (see \cite{Medkova2018}, Corollary 7.11.7) $\tilde{v}\in C(\bar{Q})$.	
 One can finally apply Lemma \ref{lem-singularities} to the function $w=v-\tilde{v}\in C(\bar{Q}\setminus \{S_{j}\})\cap C^2(Q)\cap L^2(Q)$. Since $Q$ is convex, we conclude that $w=0$, and consequently that $v|_{Q}=\tilde{v}\in H^1(Q)$. As $v\in C^\infty \left(\overline{\Omega^j\cap \Omega^{j+1}}\setminus \{S_{j}\}\right)$, we deduce that $v$ is $H^1$ in all open bounded subdomains of $\Omega^j\cap\Omega^{j+1}$. Since also $v\in L^2(\Omega^j\cap\Omega^{j+1})$, satisfies the Helmholtz equation in $ \Omega^j\cap\Omega^{j+1}$ and vanishes on $\partial(\Omega^j\cap \Omega^{j+1})$, Lemma \ref{lem-L2givesH1} implies that $v\in H^1(\Omega^j\cap\Omega^{j+1})$. Thus $v=0$ so that \eqref{eq:compatibility} holds.

\vspace{2mm}

Step 2. The second step consists in proving that, if $g=0$, the function   defined by
$	u=U^j(\varphi^j)$ in $\Omega^j$ is equal to 0 everywhere, and deduce from this that each $\varphi^j=0$. For $j=1,...,\NO$, as above let $S_j:=\Sigma^j\cap \Sigma^{j+1}$ so that   $P:=\{S_1, S_2,...,S_\mathcal{N}\}$ is the set of corners of $\partial \Omega$ (i.e. the set of vertices of the polygon $\mathcal O$). Since, for each $j$,  $\varphi^j$ is such that $\varphi^j=g=0$ on $\Gamma^j$, and $\varphi^j\in C(\Sigma^j\setminus P)$ (see Step 1), it follows from Lemma \ref{lem:UjinL2} that $u\in C^\infty(\Omega^j)\cap L^2(\Omega^j)\cap C(\overline{\Omega^j}\setminus P)${, $u = \varphi^j$ on $\Sigma^j\setminus P$, in particular} $u={\varphi^j=}0$ on $\Gamma^j$, and $\Delta u+k^2u=0$ in $\Omega^j$. Consequently $u\in C^\infty(\Omega)\cap L^2(\Omega)\cap C(\overline{\Omega}\setminus P)$, $\Delta u+k^2u=0$ in $\Omega$, and $u=0$ on $\partial \Omega\setminus P$.

Again, to conclude that $u=0$, we only need to prove that $u\in H^1(\Omega)$. We proceed as in the first step of the proof. Introducing the bounded polygon $\Omega':=\mathcal{O}'\setminus \mathcal{O}$,
we denote by $\tilde{u}$ the unique solution in  $H^1(\Omega')\cap C(\overline{\Omega'})\cap C^2(\Omega')$ of the  Dirichlet problem
$$\begin{array}{|ll}
	-\Delta \tilde{u}=k^2u&\mbox{in }\Omega',\\
	\tilde{u}=u&\mbox{on }\partial \Omega'.
\end{array}$$
Then we apply Lemma \ref{lem-singularities} to the function $w=u-\tilde{u}\in C(\overline{\Omega'}\setminus P)\cap C^2(\Omega')\cap L^2(\Omega')$. But{, in contrast} to the {argument in the } first step, the polygon $\Omega'$ is not convex and has $\mathcal{N}$ reentrant corners at the vertices $S_1$, $S_2$, ..., $S_\mathcal{N}$. Consequently, we deduce from Lemma \ref{lem-singularities}  that there exists a function $\tilde{w}\in H^1(\Omega')$ and $\mathcal{N}$ complex coefficients $c_1,..., c_\mathcal{N}$ such that
	\[w=\tilde{w}+\sum_{j=1}^\mathcal{N} c_j\,r_j^{-\pi/\alpha_j}\sin(\pi\theta_j/\alpha_j)\quad\mbox{ in }\Omega',\]
where, as explained in the lemma, $\alpha_j$ is the interior angle of $\Omega'$ at $S_j$ and $(r_j,\theta_j)$ are  polar coordinates centered at $S_j$. The definition of $\theta_j$ is such that $ \sin(\pi\theta_j/\alpha_j)$ vanishes on $\Sigma^j\cap\partial\mathcal{O}$ and $\Sigma^{j+1}\cap\partial\mathcal{O}$ but note that it does not vanish on $\Sigma^j\cap\Omega^{j+1}$. Since $\varphi^{j}\in L^2(\Sigma^{j})$, $\tilde{u}\in C(\overline{\Omega'})$, and  $w=u-\tilde{u}=\varphi^{j}-\tilde{u}$ on $\Sigma^{j}\cap \Omega'$, we must have $w\in L^2(\Sigma^{j}\cap \Omega')$, which is possible only if $c_j=0$. Indeed, since $\alpha_j<2\pi$, we have $\pi/\alpha_j>1/2$ so that $r_j^{-\pi/\alpha_j}$ is not square integrable near $r_j=0$. Summing up, we have $w=\tilde{w}$ and $u=\tilde{u}+\tilde{w}\in H^1(\Omega')$.
Since $u \in {C}^\infty(\Omega)$, one concludes that $u$ is $H^1$ in all open bounded subdomains of $\Omega$,  and finally, thanks to Lemma \ref{lem-L2givesH1}, that $u\in H^1(\Omega)$.
We have shown that $u\in H^1(\Omega)$ and $u$ satisfies \eqref{pb:probleme_Dir_diss} with $g=0$. Thus $u=0$ and, since $\varphi^j=u$ on $\Sigma^j\setminus P$, we have that $\varphi^j=0$ for all $j$, which ends the proof.

 \end{proof}
\begin{figure}[ht]
	\centering
		\begin{tikzpicture}[scale=0.5]
			\fill[rotate=-30,color=gray!20] (-3,-1)--(3.,-1)--(0.,4.2)--cycle;
			\draw[rotate=-30] (-4.5,0) -- (4.5,0);
			\draw[rotate=-30] ({sqrt(3)+1},-3) -- ({1-5/3*sqrt(3)}, 5);
			\draw[rotate=-30] ({-sqrt(3)-1},-3) -- ({5/3*sqrt(3)-1}, 5);
			\draw[rotate=-30,densely dotted] (0,{sqrt(3)+0.15}) -- (.6,{1.6*sqrt(3)+0.15});
			\draw[rotate=-30,densely dotted] (0,{sqrt(3)+0.15}) -- (-.6,{1.6*sqrt(3)+0.15});
			\draw[rotate=-30,densely dotted]  (.6,{1.6*sqrt(3)+0.15})--(0,4.05);
			\draw[rotate=-30,densely dotted] (-.6,{1.6*sqrt(3)+0.15})--(0,4.05);
			\draw[rotate=-30,densely dotted] (0,{sqrt(3)-0.15}) -- (-.85,{.15*sqrt(3)-0.15});
			\draw[rotate=-30,densely dotted] (0,{sqrt(3)-0.15}) -- (.85,{.15*sqrt(3)-0.15});
				\draw[rotate=-30,densely dotted] (-.85,{.15*sqrt(3)-0.15}) -- (.85,{.15*sqrt(3)-0.15});
			\node at (0.3, 0.5) {$\mathcal{O}$};
					\node at (1.4, 2.4) {$Q$};
			\node at (1.3, -3.5) {$\Sigma^1$};
			\node at (-3.5, -0.5) {$\Sigma^2$};
			\draw (3,4.5) node  {${\Omega^1}\cap{\Omega^2}$};
		\end{tikzpicture}
	\caption{Illustration for the proof of Theorem \ref{thm:uniqueness-kcomplex} in  the case $j=1$ and $\NO=3$. The convex polygon  $\mathcal{O}'$ is indicated in gray and  the boundaries of $\mathcal{O}$ and $Q$ with   dotted contours.}
	\label{fig:illustration-theorem}
\end{figure}
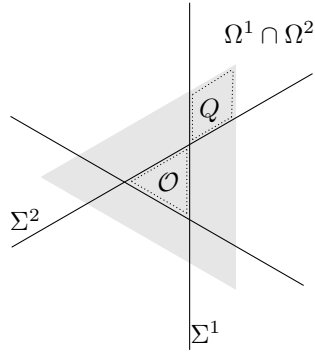
\section{The HSM method for real wavenumber}\label{sec:HSM_real}
\label{section-HSMrealfreq}
\subsection{The Half-Space Matching formulation}\label{sub-HSMrealfreqformulation}
In this section we consider the Dirichlet problem of the previous section, but now with a real wavenumber  $ k>0$.  We seek $u\in H^1_{\mathrm{loc}}(\Omega)$ 
such that
\begin{equation}\label{pb:probleme_Dir_real}
	\begin{cases}
			-\Delta u -  k^2\,u  = 0 \; \text{in} \; \Omega,\\[4pt]
			u=g\;\text{on}\;\Gamma,
			\end{cases}
\end{equation}
for a given $g\in H^{1/2}(\Gamma)$, and such that  $u$ satisfies the Sommerfeld radiation condition \eqref{eq:Sommerfeld}, which it is convenient to rewrite as  
%
	\begin{equation} \label{eq:Sommerfeld-sup}
	\lim_{R\to +\infty} R^{1/2}\sup \left\{\left|\frac{\partial u(\bsx)}{\partial r} - \ri k u(\bsx)\right|; |\bsx| = R\right\} =0.
\end{equation}
It is well known that this problem has a unique solution.

Let us now give a Half-Space Matching formulation of problem (\ref{pb:probleme_Dir_real}-\ref{eq:Sommerfeld-sup}).
If we set as before $\varphi^j:=u\big|_{\Sigma^j}$ for $j=1,...,\NO$, one can easily check by a Green's identity argument that the half-plane representations \eqref{eq:def_traces2}, with $U^j(\varphi^j)$ given by \eqref{eq:hprGreen} (equivalently \eqref{eq:hprPhi} or \eqref{eq:hpr}), still hold.
%
A main difference with the case of a complex wavenumber is that now the solution of (\ref{pb:probleme_Dir_real}-\ref{eq:Sommerfeld-sup}) decays only slowly, like $r^{-1/2}$ as $r\to \infty$, so that we cannot expect that $\varphi^j\in L^2(\Sigma^j)$. However $\varphi^j\in L^2(\Sigma^j)+L^\infty(\Sigma^j)$ so that, for $x^j_1>a^j$, the integral
$$  \int_{\R}\mathcal{H}(k;x_1^j-a^j,x_2^j-y_2^j)\,\varphi^j(a^j,y_2^j)\, \rd y_2^j$$
is still well-defined by Lemma \ref{lem:UjinL2}. 
Then one can check, exactly as in the case of a complex wavenumber, that the $\varphi^j$'s satisfy the following Half-Space Matching equations:
\begin{equation}\label{eq:HSMM_real}
	\begin{array}{c}
		\begin{array}{|ll}
			\varphi^j=D_{j-1,j}\,\varphi^{j-1},& \text{on}\;\Sigma^j\cap\Omega^{j-1},\\
			\varphi^j=g|_{\Gamma^j},& \text{on}\;\Gamma^j,\\
			\varphi^j=D_{j+1,j}\,\varphi^{j+1},& \text{on}\;\Sigma^j\cap\Omega^{j+1},
		\end{array}\quad j\in\Z/\mathcal{N}\Z
	\end{array}
\end{equation}
where
\begin{equation}\label{eq:def_DtD_real} D_{j,j\pm1}\,\varphi^{j}:=U^j(\varphi^{j})\big|_{\Omega^j\cap\Sigma^{j\pm 1}},
	\end{equation}
with $U^j(\varphi^j)$ given by \eqref{eq:hprGreen}.
The idea is again, for numerical purposes, to replace problem (\ref{pb:probleme_Dir_real}-\ref{eq:Sommerfeld-sup}) by the system of equations \eqref{eq:HSMM_real}. But  in contrast to the case with $\Im(k)>0$, it is not clear in which space the functions $\varphi^j$'s must be sought  to make \eqref{eq:HSMM_real} uniquely solvable. In what follows, we will provide a framework which ensures existence and uniqueness of the solution of \eqref{eq:HSMM_real}. More precisely, in this framework,  the unique solution of \eqref{eq:HSMM_real} is given by  $\varphi^j:=u\big|_{\Sigma^j}$, where $u$ is the unique solution of  (\ref{pb:probleme_Dir_real}-\ref{eq:Sommerfeld-sup}).

Let us first recall{, as just remarked above,} that $U^j(\varphi)$ is well-defined when $\varphi\in L^2(\Sigma^j)+L^\infty(\Sigma^j)$ by Lemma \ref{lem:UjinL2}, and{, indeed,} we will require that the solution of \eqref{eq:HSMM_real} satisfies $\varphi^j\in L^2(\Sigma^j)+L^\infty(\Sigma^j)$ for $j=1,\ldots,\NO$.
But, as discussed in \S\ref{sec:main results}, our uniqueness result requires an additional property:
analogously to requiring that $u\in H^1_{\mathrm{loc}}(\Omega)$ satisfies the Sommerfeld radiation condition when $k$ is real,   we will require, for $j= 1,...,\NO$, that $\varphi^j$    has the following asymptotics playing the role of a radiation condition:
\begin{equation} \label{eq:RadCond}
\varphi^j(a^j,t) = \left\{\begin{array}{cc} c_+^j\,\re^{\ri k |t|}|t|^{-1/2}\left(1 + O(|t|^{-1})\right), & \mbox{as} \quad t\to +\infty,\\
	c_-^j\,\re^{\ri k |t|}|t|^{-1/2}\left(1 + O(|t|^{-1})\right), & \mbox{as} \quad t\to -\infty,\end{array}\right.
\end{equation}
for some constants $c_\pm^j\in\C$. This obviously holds when $\varphi^j:=u\big|_{\Sigma^j}$, where $u$ is the unique solution of  (\ref{pb:probleme_Dir_real}-\ref{eq:Sommerfeld-sup}), since{, as discussed in \S\ref{sec:main results},  $u$ satisfies \eqref{eq:ffp}.}

The main result of the present paper is Theorem \ref{thm:HSMMUnique} stating that the radiation condition \eqref{eq:RadCond} is strong enough to ensure uniqueness for \eqref{eq:HSMM_real} for $k$ real. We will prove it by following the same steps as in the case of a complex wavenumber, proving first the compatibility of half-plane representations  in the overlaps between half-planes, and then uniqueness for the reconstructed solution in the exterior domain $\Omega$.
To establish these results, we will use classical uniqueness results for the Helmholtz equation {on unbounded domains when the solution satisfies} the Sommerfeld radiation condition at infinity. A preliminary question is: can we prove, using  \eqref{eq:RadCond},  that the half-plane representations \eqref{eq:hpr} satisfy a Sommerfeld condition at infinity?  We will establish, in the next subsection, a slightly weaker result which will be sufficient for the proof of Theorem \ref{thm:HSMMUnique}.

\subsection{Properties of the half-plane solution} \label{sec:half-plane}


We consider in this subsection  the half-space $\R^2_+:=(0,\infty)\times \R$.
Let us prove a preliminary lemma for the solution of the Dirichlet problem in $\R^2_+$ given by \eqref{eq:hplem-1}, in the case that the Dirichlet data decays sufficiently rapidly at infinity.
\begin{lem}
	\label{lem-decayingphi}
	Suppose that $\phi\in L^\infty(\R)$ satisfies, for  some $p>1$,
	\begin{equation} \label{eq:phipsi}
		|\phi(t)| \leq (1+|t|)^{-p}, \quad t\in \R,
	\end{equation}
	and define $u(\bsx)$, for $\bsx\in \R^2_+$, by
	\begin{equation} \label{eq:hplem}
		u(\bsx) := \int_\R \mathcal{H}(k;x_1,x_2-t)\phi(t)\, \rd t.
	\end{equation}
	Then there exists a constant $C(p)>0$  such that
	\begin{equation} \label{eq:ubound}
		|u(\bsx)| \leq C(p) (1+|\bsx|)^{-1/2}, \quad \bsx\in \R^2_+.
	\end{equation}
\end{lem}
\begin{proof}
	We will use the bound \eqref{eq:Hbound},
	where  $C>0$, there and throughout the remainder of the proof, denotes a constant independent of $\bsx$ (and also independent of $p$), not necessarily the same at each occurrence.
	It will be convenient for the proof to denote by $u[\phi]$  the function $u$ defined by \eqref{eq:hplem}.
	Then defining $\chi_{x_2}\in L^\infty(\R)$, for $x_2\in \R$, to be the characteristic function of the interval
	$$
	I_{x_2} :=  \left\{t\in \R;|x_2-t|\leq (1+|x_2|)/2\right\},
	$$
	we have  $u[\phi]=u[\chi_{x_2}\phi]+u[(1-\chi_{x_2})\phi]$. We will establish the bound \eqref{eq:ubound} for each term of the sum  (cf.\ the argument between (4.16) and (4.17) in \cite{ACWH:02}).
	
	For the first term we have
	$$|u[\chi_{x_2}\phi](\bsx)|\leq \int_\R|\mathcal{H}(k;x_1,x_2-t)|\chi_{x_2}(t)|\phi(t)|\rd t,$$
	which gives
	\begin{equation}
		\label{eq:bound-uchiphi}
		|u[\chi_{x_2}\phi](\bsx)|\leq 2\|\chi_{x_2}\phi\|_{L^\infty(\R)}\int_0^{\frac{1+|x_2|}{2}}|\mathcal{H}(k;x_1,s)|\rd s .
	\end{equation}
Now note that  if $|t|< |x_2|/3$ and $|x_2|\geq 3$, then $|t-x_2|>2|x_2|/3\geq (1+|x_2|)/2$, so that $\chi_{x_2}(t) = 0$. Thus, for $|x_2|\geq 3$,  \eqref{eq:phipsi} implies that
	$$
	\|\chi_{x_2}\phi\|_{L^\infty(\R)} \leq \sup_{|t|\geq|x_2|/3}|\phi(t)|\leq (1+|x_2|/3)^{-p},
	$$
	while also $\|\chi_{x_2}\phi\|_{L^\infty(\R)}\leq \|\phi\|_{L^\infty(\R)} \leq 1$, for every $x_2\in \R$, so  that
	\begin{equation} \label{eq:chiphi2}
		\|\chi_{x_2}\phi\|_{L^\infty(\R)} \leq C(1+|x_2|)^{-p}, \quad x_2\in \R.
	\end{equation}
Moreover, it follows from \eqref{eq:Hbound}   that
	\begin{eqnarray*}
		\int_0^{\frac{1+|x_2|}{2}}|\mathcal{H}(k;x_1,s)|\rd s &\leq
		& Cx_1	\int_0^{\frac{1+|x_2|}{2}}\left[(x_1^2+s^2)^{-1}+(x_1^2+s^2)^{-3/4}\right]\rd s\\
		& =& C\int_0^{\frac{1+|x_2|}{2x_1}}\left[(1+t^2)^{-1}+x_1^{1/2}(1+t^2)^{-3/4}\right]\rd t\\
		& \leq &\displaystyle C(1+x_1)^{1/2}\int_0^{\frac{1+|x_2|}{2x_1}}(1+t^2)^{-3/4}\rd t\\
		& \leq &\displaystyle C(1+x_1)^{1/2}\min\left(1,\frac{1+|x_2|}{2x_1}\right).
	\end{eqnarray*}
	Combining this with \eqref{eq:bound-uchiphi} and \eqref{eq:chiphi2}, we get the following estimate:
	$$	|u[\chi_{x_2}\phi](\bsx)|\leq  C(1+|x_2|)^{-p}(1+x_1)^{1/2}\min\left(1,\frac{1+|x_2|}{2x_1}\right),$$
	which gives the expected bound (recalling that $p>1$)
	$$	|u[\chi_{x_2}\phi](\bsx)|\leq  C(1+|\bsx|)^{-1/2},\quad\bsx\in\R^2_+.$$
	
	To see that the same bound holds for $u[(1-\chi_{x_2})\phi]$, we note that
	$$ |u[(1-\chi_{x_2})\phi](\bsx)|\leq C\sup_{t\in\R}|\mathcal{H}(k;x_1,x_2-t)(1-\chi_{x_2})(t)|\int_{\R} |\phi(t)|\, \rd t.$$
	Observing that, for $\bsx\in \R^2_+$ and $t\in \R\setminus I_{x_2}$,
	$$
	(x_1^2+(x_2-t)^2)^{1/2} \geq (x_1^2+(1+|x_2|)^2/4)^{1/2}\geq C(1+|\bsx|),
	$$
	and  using \eqref{eq:Hbound} and \eqref{eq:phipsi}, it follows that, for $\bsx\in\R^2_+$,
	\begin{equation}
				|u[(1-\chi_{x_2})\phi](\bsx)|\leq C\frac{x_1}{(1+|\bsx|)^{3/2}}\int_{\R} |\phi(t)|\, \rd t\leq \frac{C}{p-1}(1+|\bsx|)^{-1/2}.
			\end{equation}
\end{proof}

It follows from standard interior elliptic regularity results that if $u\in {C}^2(\R^2_+)$ satisfies
$$\Delta u +k^2u=0\quad\mbox{in }\R^2_+,$$
then for every $\varepsilon>0$ and $\bsx\in\R^2_+$ with $x_1>\varepsilon$, so that $\{\bsy\in\R^2;|\bsx-\bsy|\leq\varepsilon\}\subset \R^2_+$:
\begin{equation}
	\label{eq:elliptic-gradient}
	|\nabla u(\bsx)|\leq C\:\frac{1+k^2\varepsilon^2}{\varepsilon}\max_{|\bsx-\bsy|\leq\varepsilon}|u(\bsy)|
\end{equation}
where $C$ is an absolute constant (see, e.g. \cite[Lemma 2.7]{CWZ:98}).
Combining this estimate with the previous lemma, we have immediately the following corollary:
\begin{cor}\label{cor-nablaubound}
	Suppose that $\phi\in L^\infty(\R)$ satisfies \eqref{eq:phipsi} for  some $p>1$. Then there exists a constant $C(p)>0$ such that for all $\varepsilon>0$,
 the function $u$ defined by \eqref{eq:hplem} satisfies
	\begin{equation} \label{eq:nablaubound}
		|\nabla u(\bsx)| \leq C(p)\:\frac{1+k^2\varepsilon^2}{\varepsilon} (1+|\bsx|)^{-1/2}, \quad  x_1>\varepsilon, \,x_2\in\R.
	\end{equation}
\end{cor}

Now, using the previous results, we will establish a radiation condition  for $u$ defined by \eqref{eq:hplem} when its trace $\varphi$ satisfies itself a radiation condition.
\begin{lem}
	Suppose that $\phi\in {C}(\R)$ satisfies
	\begin{equation} \label{eq:phi_asymp}
		\phi(t) = \left\{\begin{array}{cc} c_+\,\re^{\ri k |t|}|t|^{-1/2} + O(|t|^{-p}), & \mbox{as} \quad t\to +\infty,\\
			c_-\,\re^{\ri k |t|}|t|^{-1/2} + O(|t|^{-p}), & \mbox{as} \quad t\to -\infty,\end{array}\right.
	\end{equation}
	for some constants $c_+,c_-\in \C$ and   $p>1$. Define $u$ by \eqref{eq:hplem} in $\R^2_+$. Then the following properties hold:
	\begin{itemize}
		\item[(i)]  there exists a constant $C>0$ such that
		\begin{equation} \label{eq:ubound-general}
			|u(\bsx)| \leq C (1+|\bsx|)^{-1/2}, \quad \bsx\in \R^2_+;
		\end{equation}
		\item[(ii)]  for all $\varepsilon>0$,
		\begin{equation} \label{eq:Sommerfeld-epsilon}
			\lim_{R\to +\infty} R^{1/2}\sup \left\{\left|\frac{\partial u(\bsx)}{\partial r} - \ri k u(\bsx)\right|; x_1>\varepsilon \mbox{ and }|\bsx| = R\right\} =0.
		\end{equation}
	\end{itemize}
\end{lem}
\begin{proof}
	We use the same notation $u[\phi]$ for \eqref{eq:hplem} as in the proof of Lemma \ref{lem-decayingphi}. We prove this proposition by writing  $\phi$ as $\phi=\phi^{(1)}+\phi^{(2)}$, and correspondingly writing the integral \eqref{eq:hplem} as $u[\phi]=u[\phi^{(1)}]+u[\phi^{(2)}]$, choosing this splitting as follows: $\phi^{(1)}$ is the trace on the boundary of a solution   of the Helmholtz equation in $\R^2_+$ that clearly satisfies \eqref{eq:ubound-general} and \eqref{eq:Sommerfeld-epsilon}, and $\phi^{(2)}(t)$ decays sufficiently rapidly as $|t|\to\infty$ so that, thanks to Lemma \ref{lem-decayingphi} and Corollary \ref{cor-nablaubound},  $u[\phi^{(2)}]$ also satisfies   \eqref{eq:ubound-general} and \eqref{eq:Sommerfeld-epsilon}.
	
	In more detail, from the asymptotics  of the Hankel function $H_0^{(1)}$ it follows that, for some non-zero $c^*\in \C$,
	\begin{equation} \label{eq:Pasymp}
		\Phi(\bsx,(0,t))= c^*\,\re^{\ri k |t-x_2|}\,|t|^{-1/2} + O(|t|^{-3/2}) \quad \mbox{as} \quad |t|\to \infty,
	\end{equation}
	for every $\bsx\in \R^2$.
	Pick $\bsz=(z_1,z_2),\bsz'=(z'_1,z'_2)\in \R^2$ with $k(z_2-z'_2)/\pi\not\in \Z$, and with $z_1,z'_1<0$.  Define $\phi^{(1)}\in C(\R)$ by
	\begin{equation}\label{eq:phi_1}
		\phi^{(1)}(t) := (c/c^*) \Phi(\bsz,(0,t))+(c'/c^*)\Phi(\bsz',(0,t))\;\text{ for}\; t\in \R.
		\end{equation}
		 Then $\phi^{(1)}(t)$ has the same leading asymptotic behaviour as $\phi(t)$ as $t\to\pm\infty$ provided
	\begin{equation}\label{eq:phi_1_const}
	c\re^{-\ri kz_2} + c'\re^{-\ri kz'_2} = c_+ \quad \mbox{and} \quad c\re^{\ri kz_2} + c'\re^{\ri kz'_2} = c_-.
	\end{equation}
Let us choose $c$ and $c'$ so that these two linear constraints hold; this is possible since $k(z_2-z'_2)/\pi\not\in \Z$ implies that the determinant $\exp(\ri k(z'_2-z_2))-\exp(\ri k(z_2-z'_2)=2\ri \sin(k(z_2-z'_2))\neq 0$ . 	Then it is clear from an application {of Green's second identity } that,
	for $\bsx\in \R^2_+$,  $u[\phi^{(1)}](\bsx)=(c/c^*) \Phi(\bsz,\bsx)+(c'/c^*)\Phi(\bsz', \bsx)$ (see, e.g., the proof that (ii)$\Rightarrow$(v)$\Rightarrow$(i) in \cite[Theorem 2.9]{CWZ:98}).
	Hence $u[\phi^{(1)}]$ satisfies items (i) and (ii).
	
	It remains to show that the same is true for $u[\phi^{(2)}]$ where $\phi^{(2)}=\phi-\phi^{(1)}$. It follows  from \eqref{eq:phi_asymp} and \eqref{eq:Pasymp}--\eqref{eq:phi_1_const} that,  for some $c_0>0$, where  $p':=\min(p,3/2)$,
	\begin{equation} \label{eq:phipsi2}
		|\phi^{(2)}(t)| \leq c_0(1+|t|)^{-p'}, \quad t\in \R.
	\end{equation}
	Item (i) is therefore a direct consequence of Lemma \ref{lem-decayingphi}. To see that $u[\phi^{(2)}]$ also satisfies the radiation condition (ii), choose $q\in (1,p')$ and,
	for each $A>0$, let $\chi_A$ be the characteristic function of the interval $[-A,A]$. Write  $u[\phi^{(2)}] =u[\chi_A\phi^{(2)}]+ u[(1-\chi_A)\phi^{(2)}]$.
	Now, by \eqref{eq:phipsi2}, $$|(1-\chi_A(t))\phi^{(2)}(t)| \leq c_0 (1+|t|)^{-q}(1+A)^{q-p'}$$ so that, by \eqref{eq:ubound} and \eqref{eq:nablaubound} (applied with $p$ replaced by $q$),  there exists a constant $C(q)>0$ depending only on $q$ such that
	$$
	|u[(1-\chi_A)\phi^{(2)}](\bsx)| \leq C(q)c_0(1+A)^{q-p'}(1+|\bsx|)^{-1/2}, \quad \bsx\in \R^2_+,
	$$
	and from Corollary \ref{cor-nablaubound}, for all $\varepsilon>0$, there exists a constant $C(q,\varepsilon)>0$ depending only on $q$ and $\varepsilon$ such that
	$$
	|\nabla u[(1-\chi_A)\phi^{(2)}](\bsx)| \leq C(q,\varepsilon)c_0(1+A)^{q-p'}(1+|\bsx|)^{-1/2}, \quad x_1>\varepsilon, \,x_2\in\R.
	$$
	Let us set for any function $v$ defined on $\R^2_+$:$$M_{R,\varepsilon}(v) := R^{1/2}\sup \left\{\left|\frac{\partial v(\bsx)}{\partial r} - \ri k v(\bsx)\right|; x_1>\varepsilon \mbox{ and }|\bsx| = R\right\} .
	$$
	It  follows  from the  estimates above that, for all $R>0$,
	$$M_{R,\varepsilon}(u[(1-\chi_A)\phi^{(2)}])\leq \max(C(q),C(q,\varepsilon))c_0(1+A)^{q-p'}.$$
	Thus, for all $\varepsilon>0$, given $\eta>0$ we can choose $A>0$ such that $M_{R,\varepsilon}(u[(1-\chi_A)\phi^{(2)}])\leq \eta/2$ for all $R>0$. But also, for each $A>0$, it is standard that $u[\chi_A\phi^{(2)}](\bsx)$, a double-layer potential supported on a bounded interval, satisfies the Sommerfeld radiation condition, in particular that $M_{R,\varepsilon}(u[\chi_A\phi^{(2)}])\leq \eta/2$ for $R$ large enough. Thus, $M_{R,\varepsilon}(u[\phi^{(2)}])\leq \eta$ for all sufficiently large $R$, which concludes the proof.
\end{proof}

\begin{cor}
	\label{cor:Helmholtz-halfplane}
	Suppose that $\phi=\phi_1+\phi_2$ where $\phi_1\in L^2(\R)$ is compactly supported and $\phi_2\in {C}(\R)$. If $\phi$ satisfies
	\eqref{eq:phi_asymp}  and $u$  is defined by \eqref{eq:hplem}, then
	\begin{itemize}
		\item[(i)]  there exist two constants $C>0$ and $R>0$ such that
		\begin{equation} \label{eq:ubound-largeR}
			|u(\bsx)| \leq C (1+|\bsx|)^{-1/2}, \quad \bsx\in \R^2_+, \; |\bsx|>R;
		\end{equation}
		\item[(ii)]  for all $\varepsilon>0$,
		\eqref{eq:Sommerfeld-epsilon} holds.
	\end{itemize}
\end{cor}
\begin{proof}
	Again, we write $u[\phi]=u[\phi_1]+u[\phi_2]$. The properties for $u[\phi_2]$ have been established in the previous lemma, while $u[\phi_1]$ is a double-layer potential supported on a bounded interval, for which such properties are standard, e.g. \cite[Theorem 2.14, Lemma 2.5]{ChGrLaSp:11}.
\end{proof}

\subsection{Proof of the uniqueness result}
We are now able to prove the first main result of the paper.

\begin{theorem} \label{thm:HSMMUnique}
Suppose $k>0$. There exists, for each $g\in H^{1/2}(\Gamma)$, a unique solution $(\varphi^1,..., \varphi^\NO)\in L^2_{\rm loc}(\Sigma^1)\times\ldots\times L^2_{\rm loc}(\Sigma^\NO)$  satisfying \eqref{eq:HSMM_real} such that, for all $j$, $\varphi^j$ satisfies the radiation condition \eqref{eq:RadCond}.
\end{theorem}
Existence was proven by the construction above of a solution of \eqref{eq:HSMM_real} from the unique solution $u$ of (\ref{pb:probleme_Dir_real},\ref{eq:Sommerfeld}), namely $\varphi^j:=u|_{\Sigma^j}$, $j=1,...,\NO$. Note in particular that the Sommerfeld radiation condition \eqref{eq:Sommerfeld} for $u$ implies the far-field asymptotics \eqref{eq:ffp}, which implies in turn that each trace $\varphi^j$ satisfies the radiation condition \eqref{eq:RadCond}.
Uniqueness follows from the following proposition since (\ref{pb:probleme_Dir_real}, \ref{eq:Sommerfeld}) is well-posed. 
\begin{proposition}\label{prop:HSMMUnique}
	Suppose $k>0$ and let $ (\varphi^1,..., \varphi^\NO)\in L^2_{\rm loc}(\Sigma^1)\times\ldots\times L^2_{\rm loc}(\Sigma^\NO)$  satisfying \eqref{eq:HSMM_real} (with $g\in H^{1/2}(\Gamma)$) be such that, for all $j$, $\varphi^j$ satisfies the radiation condition \eqref{eq:RadCond}.
	Then the compatibility equations \eqref{eq:compatibility} hold  for all $j\in \Z/\mathcal{N}\Z$,
	 the function  defined by
	$	u=U^j(\varphi^j)$ in $\Omega^j$, $j=1,...,\NO$,
	is the unique solution $u\in H^1_{\rm loc}(\Omega)$ of (\ref{pb:probleme_Dir_real}, \ref{eq:Sommerfeld}), and $\varphi^j=u|_{\Sigma^j}$ for each $j$.
	\end{proposition}

\begin{proof}{For each $j$,
		 $\varphi^j$ can be decomposed as $\varphi^j=\varphi^j_1+\varphi^j_2$ with $\varphi^j_1=\varphi^j\chi_A^j$ and $\varphi^j_2=\varphi^j(1-\chi_A^j)$, $\chi_A^j$ being the characteristic function of $\{(x_1^j=a^j,x_2^j);x_2^j\in(-A,A)\}$. Since $\varphi^j\in L^2_{\rm loc}(\Sigma^j)$ satisfies \eqref{eq:RadCond},  we have that $\varphi^j_1\in L^2(\Sigma^j)$ with compact support and, for $A$ large enough, $\varphi^j_2\in L^\infty(\Sigma^j)$. From Lemma \ref{lem:UjinL2}, and by the definition \eqref{eq:def_DtD_real} of the operator $D_{j,j\pm1}$, since $(\varphi^1,..., \varphi^\NO)$ satisfies \eqref{eq:HSMM_real}, we deduce, provided we choose $A$ large enough, that  $\varphi^j_2\in C(\Sigma^j)$} for all $j$.
	
We follow now the same steps as in the proof of Proposition \ref{prop:trraceL2}. 	

The first step consists in proving equations \eqref{eq:compatibility}. Again, we introduce $v=U^j(\varphi^j)-U^{j+1}(\varphi^{j+1})$ and one can show by using Lemma \ref{lem:UjinL2} as in the proof of Proposition \ref{prop:trraceL2} that $v\in L^2_{\mathrm{loc}}(\Omega^j\cap\Omega^{j+1})\cap C^\infty(\Omega^j\cap\Omega^{j+1})\cap C(\overline{\Omega^j\cap\Omega^{j+1}}\setminus\{S_j\})$, satisfies $\Delta v+k^2v=0$ in $\Omega^j\cap\Omega^{j+1}$, and vanishes on  $\partial(\Omega^j\cap\Omega^{j+1})\setminus\{S_j\}$. Moreover, proceeding exactly as in the proof of Proposition \ref{prop:trraceL2},
one can show that  $v\in H^1_{\mathrm{loc}}(\Omega^j\cap\Omega^{j+1})$. We would like, by applying Corollary \ref{cor:Helmholtz-halfplane} to $U^j(\varphi^j)$ and $U^{j+1}(\varphi^{j+1})$, to deduce that $v$ also satisfies the Sommerfeld radiation condition in $\Omega^j\cap \Omega^{j+1}$. The issue  is that Corollary \ref{cor:Helmholtz-halfplane} does not ensure that $v$ satisfies the Sommerfeld radiation condition up to the boundary. But this difficulty can be overcome; since $v=0$ on  $\partial(\Omega^j\cap\Omega^{j+1})$, reflection arguments and standard elliptic regularity results  combined with the estimate (\ref{eq:ubound-largeR}) imply that, for $R$ large enough, there exists a constant $C>0$ such that:
\begin{equation}
	\label{eq:boundsforv}
	|v(\bsx)|+|\nabla v(\bsx)|\leq C (1+|\bsx|)^{-1/2}, \quad \bsx\in \Omega^j\cap\Omega^{j+1}, |\bsx|>R.
\end{equation}
Then, denoting  $\Sigma^j_R:=\{\bsx\in \Omega^j\cap\Omega^{j+1};|\bsx|=R\}$ and $\dsp d^j(\bsx):=\min_{\bsy \in \partial(\Omega^j\cap\Omega^{j+1})}|\bsx-\bsy|$ , one has by \eqref{eq:boundsforv} for all $R$ large enough,
$$ \int_{\bsx\in\Sigma^j_R, d^j(\bsx)<1}\left|\frac{\partial v(\bsx)}{\partial r} - \ri k v(\bsx)\right|^2ds(\bsx)\leq C(1+R)^{-1},$$
 where the constant $C$ is independent of $R$. Combined with \eqref{eq:Sommerfeld-epsilon} with $\varepsilon=1$ applied to
  $U^j(\varphi^j)$ and $U^{j+1}(\varphi^{j+1})$, this proves that
\begin{equation}\label{eq:Sommerfeld-integrated}
 \lim_{R\to +\infty}\int_{\Sigma^j_R}\left|\frac{\partial v(\bsx)}{\partial r} - \ri k v(\bsx)\right|^2ds(\bsx)=0,	
\end{equation}
i.e. $v$ satisfies the Sommerfeld radiation condition in a standard integrated form.
Since $v\in H^1_{\rm loc}(\Omega^j\cap\Omega^{j+1})$  satisfies $\Delta v+k^2v=0$ in $\Omega^j\cap\Omega^{j+1}$, the Sommerfeld radiation condition \eqref{eq:Sommerfeld-integrated}, and vanishes on  $\partial(\Omega^j\cap\Omega^{j+1})\setminus\{S_j\}$, one can conclude that $v=0$ by   uniqueness   of such problems in conical domains \cite{jones1953}.

The second step of the proof is simpler. Arguing as in the proof of Proposition \ref{prop:trraceL2}, it is enough to prove, in the case
 $g=0$, that the function   defined by
$	u=U^j(\varphi^j)$ in $\Omega^j$, $j=1,...,\NO$, is equal to 0 everywhere, and to deduce from this that each $\varphi^j=0$. Proceeding exactly as in that proof,
one can show that  $u\in H^1_{\mathrm{loc}}(\Omega)\cap C(\overline{\Omega}\setminus P)$, where $P$ is the set of corners of $\partial\Omega$, that $u$ is a solution of \eqref{pb:probleme_Dir_real} with $g=0$, and that $\varphi^j=u|_{\Sigma^j}$ on $\Sigma^j\setminus P$. Moreover, thanks to the overlaps between the halfspaces $\Omega^j$, one can deduce from Corollary \ref{cor:Helmholtz-halfplane} applied to each $U^j(\varphi^j)$  that $u$ satisfies the Sommerfeld condition \eqref{eq:Sommerfeld}. This implies that $u=0$ so that $\varphi^j=0$ for all $j$,
\end{proof}

Theorem \ref{thm:HSMMUnique} is a (new) well-posedness result for \eqref{eq:HSMM_real} in the case of real $k$. And, indeed, numerical solution of \eqref{eq:HSMM_real} works well \cite{BB-Fliss-Tonnoir-2018,Bon-Fli-Tja-2019} for real $k$. However, there are still significant gaps in our theoretical understanding of this formulation. In particular, while \eqref{eq:HSMM_complexe} for complex wavenumber $k$ can be written formally in operator form satisfying a Fredholm property, in the case when $k$ is real we know of no function space setting for which this formulation makes sense, where the  $D_{j,j+1}$ are well-defined bounded linear operators.

\section{General configurations}
\label{sec-CHSMgeneralcase}
Let us now recall how to extend the HSM formulation to the general problem presented in the introduction and extend the previous uniqueness result to this new formulation. More precisely, for a real wavenumber $k>0$ we consider,  as discussed in \S\ref{sec:scattering problem}, the isotropic Helmholtz equation   \eqref{eq:helmholtz-variable} in a Lipschitz domain $\Omega$,
	where $\rho-1\in L^\infty(\Omega)$ and $f\in L^2(\Omega)$ are compactly supported and $\Omega$ is $\R^2$ or $\R^2$ minus a set of Lipschitz obstacles which are supposed to lie in a bounded domain. To ensure uniqueness we impose some constraint on $\rho$, e.g.\ that $\rho$ is real-valued or that $\Im(\rho)\geq 0$. As announced in \S\ref{sec:scattering problem} the problem we consider is to seek $u\in H^1_{\mathrm{loc}}(\Omega)$ that satisfies \eqref{eq:helmholtz-variable}, the Sommerfeld radiation condition \eqref{eq:Sommerfeld}, and a boundary condition on $\Gamma= \partial \Omega$. To be specific we will impose the homogeneous Neumann condition
\begin{equation} \label{eq:neumann}
\frac{\partial u}{\partial n} = 0 \quad \mbox{on} \quad \Gamma,
\end{equation}
but the modifications to use other boundary conditions and/or include inhomogeneous terms are straightforward.

It is well known that this problem is uniquely solvable; in particular, with one of the above constraints on $\rho$, uniqueness follows by a Green's identity, the Rellich lemma, and unique continuation arguments (e.g., \cite[Theorem 8.7]{CoKr:98}).

Let $\mathcal{O}$ be   the interior of
a convex polygon containing the supports of all these perturbations.  The half-space matching method has been mainly presented for the case where $\mathcal{O}$ is a   square   (see \cite{BB-Fliss-Tonnoir-2018}). But, whereas in the problem in \S \ref{sec:HSM_real} the polygon $\O$ is given and   the  number of trace unknowns is imposed by the number of sides of   $\O$, here   we have freedom to choose the polygon $\mathcal{O}$. In particular, choosing    $\mathcal{O}$  to be  a triangle   has the advantage of minimising   the number of trace unknowns. In the sequel, we suppose that $\mathcal{O}$ is a polygon with $\mathcal{N}$ edges and we adopt  the same notations as introduced in \S \ref{section-HSMcomplexfreq} for the edges $\Gamma^1,\ldots,\Gamma^\mathcal{N}$, the angles between the edges, the $\mathcal{N}$ overlapping half-planes $\Omega^1,\ldots\Omega^\mathcal{N}$, their boundaries $\Sigma^1,\ldots,\Sigma^\mathcal{N}$, and their associated local coordinate systems.

	Let us now introduce a  bounded   Lipschitz domain $\Omega_b\subset \Omega$ containing $\overline{\mathcal{O}} \cap \Omega$   such that $\partial \Omega_b\setminus \Gamma$ is connected (for examples see Figure \ref{fig:general-configuration}), and a partition $\Gamma_b^1,\ldots,\Gamma_b^\mathcal{N}$ of   $\partial\Omega_b  \setminus \Gamma$ such that, for all $j=1,...,\NO$,   $\Gamma_b^j$ is connected and  $\Gamma_b^j\subset \Omega^j$ with $\overline{\Gamma_b^j}\cap \Sigma^j=\emptyset$. There is, of course, no unique  choice for   such a partition, but the uniqueness result   given in Theorem \ref{thm:HSMMUnique_casgeneral}   below implies that the   solution of the HSM formulation that we will write down   is independent of the choice of the partition.
	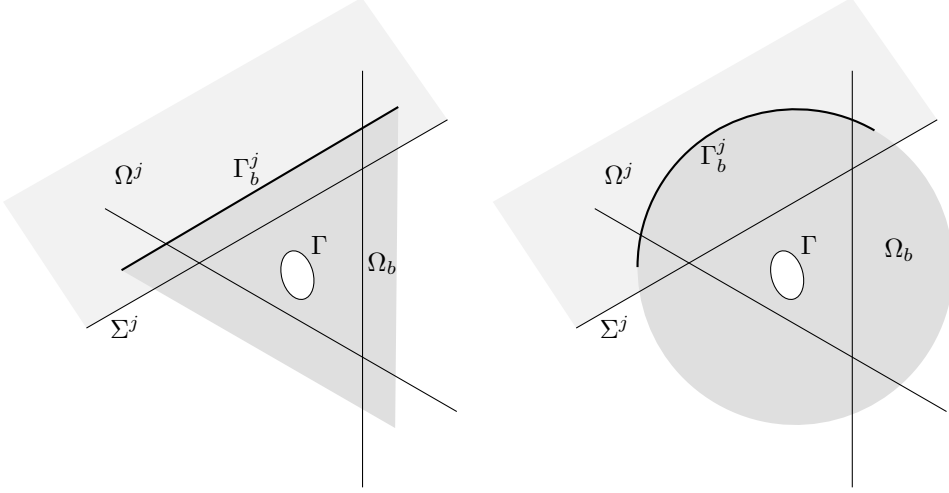
\begin{figure}[ht]
		\centering
		\begin{subfigure}{.5\textwidth}
\begin{tikzpicture}[scale=.6]
	\fill[rotate=-30,color=gray!10] ({-sqrt(3)-1.8},-3) -- ({5/3*sqrt(3)-1.8}, 5) -- ({5/3*sqrt(3)-4.8}, 6.4) --({-sqrt(3)-4.8},-1.5) -- cycle ;
    \fill[rotate=-30,color=gray!25] (-3.5,-1.5)--(3.5,-1.5)--(0.,4.7)--cycle;
    \draw[rotate=-30,thick] (-3.5,-1.5)--(0.,4.7);
	\fill[rotate=15,white] (0.2,0.3)ellipse(0.35 and 0.55);
	\draw[rotate=15] (0.2,0.3)ellipse(0.35 and 0.55);
	\draw[rotate=-30] (-4.5,-.5) -- (4.5,-.5);
	\draw[rotate=-30] ({sqrt(3)+1.8},-3) -- ({1.8-5/3*sqrt(3)}, 5);
	\draw[rotate=-30] ({-sqrt(3)-1.8},-3) -- ({5/3*sqrt(3)-1.8}, 5);
	\node at (0.6, 1) {$\Gamma$};
	\node at (2., 0.6) {$\Omega_b$};
	\node at (-1, 2.7) {$\Gamma^j_b$};
	\node at (-3.6,2.6) {$\Omega^j$};
	\node at (-3.7, -0.8) {$\Sigma^j$};
\end{tikzpicture}\end{subfigure}\begin{subfigure}{.5\textwidth}
			\begin{tikzpicture}[scale=.6]
				\fill[rotate=-30,color=gray!10] ({-sqrt(3)-1.8},-3) -- ({5/3*sqrt(3)-1.8}, 5) -- ({5/3*sqrt(3)-4.8}, 6.4) --({-sqrt(3)-4.8},-1.5) -- cycle ;
				\fill[rotate=-30,color=gray!25](0.,0.6) circle (3.5) ;
				
				\draw[rotate=-30,thick] (0,{0.6+3.5}) arc (90:210:3.5) ;
				\fill[rotate=15,white] (0.2,0.3)ellipse(0.35 and 0.55);
				\draw[rotate=15] (0.2,0.3)ellipse(0.35 and 0.55);
				\draw[rotate=-30] (-4.5,-.5) -- (4.5,-.5);
				\draw[rotate=-30] ({sqrt(3)+1.8},-3) -- ({1.8-5/3*sqrt(3)}, 5);
				\draw[rotate=-30] ({-sqrt(3)-1.8},-3) -- ({5/3*sqrt(3)-1.8}, 5);
				\node at (0.6, 1) {$\Gamma$};
				\node at (2.6, 0.9) {$\Omega_b$};
				\node at (-1.5, 3.) {$\Gamma^j_b$};
				\node at (-3.6,2.6) {$\Omega^j$};
				\node at (-3.7, -0.8) {$\Sigma^j$};
		\end{tikzpicture}\end{subfigure}
		\caption{Two examples of possible geometries for our general configuration: $\mathcal O$ is the triangle bounded by the lines $\Sigma^1$, $\Sigma^2$, and $\Sigma^3$, and  $\Omega_b$, shaded in darker gray, is a larger triangle in the left figure and a disc in the right figure.}
		\label{fig:general-configuration}
	\end{figure}

	The unknowns of the HSM formulation are the traces  $\varphi^1,..., \varphi^\NO$ of the   solution $u$   on the infinite lines $\Sigma^1,...,\Sigma^\NO$, and the restriction $u_b:=u|_{\Omega^b}$ to the   bounded domain $\Omega_b$. Let us derive the equations linking $\varphi^1,..., \varphi^\NO$ and $u_b$. On the one hand, one can show as in the case of scattering by a polygon (see \S \ref{sub-HSMrealfreqformulation}) that the $\varphi^j$'s satisfy the equations
	\begin{equation}\label{eq:HSMM_real_gen}
		\begin{array}{c}
			\begin{array}{|ll}
				\varphi^j=D_{j-1,j}\,\varphi^{j-1},& \text{on}\;\Sigma^j\cap\Omega^{j-1},\\
				\varphi^j=D_{j+1,j}\,\varphi^{j+1},& \text{on}\;\Sigma^j\cap\Omega^{j+1},
			\end{array}\quad j\in\Z/\mathcal{N}\Z,
		\end{array}
	\end{equation}
	where the operators $D_{j,j\pm1}$ are defined in \eqref{eq:def_DtD_real}. Compared to the problem of scattering by a polygon, the conditions for the traces on $\Gamma^j$ (linked to the given data $g$ in the case of scattering by a polygon) have to be replaced by an equality between $\varphi^j$ and $u_b$ on $\Gamma^j$:
	\[
		\varphi^j=u_b\;\text{on}\;\Gamma^j,\;j=1,...,\NO.
	\]
	As in the case of scattering by a polygon, since the wavenumber is real, the traces $\varphi^j$ are not in $L^2$ but they decay like $|y_2^j|^{-1/2}$ at infinity and we will require, as in \S  \ref{sub-HSMrealfreqformulation},  that, for $j=1,...,\NO$, $\varphi^j$ satisfies the asymptotic condition \eqref{eq:RadCond} at $\pm\infty$.
	
	On the other hand, we can derive a variational formulation for $u_b$ in $\Omega_b$. Since $-\Delta u_b-k^2\rho u_b=f$ in $\Omega_b$, $f$ is supported in $\mathcal{O}^\prime:= \mathcal{O}\cap \Omega$, and $u$ satisfies \eqref{eq:neumann}, the following Green's identity holds for all $v_b\in H^1(\Omega_b)$, where $n$ is the normal pointing out of $\Omega_b$:
\begin{equation}
\label{FVub1}
\int_{\Omega_b}\left(\nabla u_b\cdot \overline{\nabla v_b}-k^2 \rho u_b  \overline{v_b}\right)-\int_{\partial\Omega_b\setminus \Gamma}\frac{\partial u_b}{\partial n}\overline{v_b}\, =\int_{\mathcal{O}^\prime} f\overline{v_b}.
\end{equation}
The last idea is to replace the normal derivative on each part $\Gamma_b^j$ of $\partial\Omega_b\setminus \Gamma$ by an integral representation as a function of $\varphi^j$. Indeed, since, for all $j=1,...,\NO$, $\Gamma_b^j\subset \Omega^j$, we must have
\begin{equation}\label{eq:comp_robin}
	\frac{\partial u_b}{\partial n}-\ri ku_b=\frac{\partial U^j(\varphi^j)}{\partial n}-\ri kU^j(\varphi^j)\mbox{ on }\Gamma^j_b,\;j=1,...,\NO,
	\end{equation}
where $U^j(\varphi^j)$ is the restriction of the solution $u$ to the half-plane $\Omega^j$ and is expressed in terms of $\varphi^j$ in \eqref{eq:hprGreen}. Let us emphasize that our choice of   imposing equality of Robin traces instead of normal derivatives is so that later we have uniqueness for the formulation for all $k>0$. Note also that, since $\overline{\Gamma_b^j}\cap {\Sigma^j}=\emptyset$, the Robin trace of $U^j(\varphi^j)$ on $\Gamma^j_b$ is ${C}^\infty$ (see Lemma \ref{lem:UjinL2}). To express \eqref{eq:comp_robin} succinctly we introduce the Dirichlet-to-Robin  (DtR) operators $\Lambda^j$, given by
\begin{equation}\label{eq:def_DtR}
	\Lambda^j\varphi:=\left(\frac{\partial U^j(\varphi)}{\partial n}-\ri kU^j(\varphi)\right)\Big|_{\Gamma^j_b}.
\end{equation}
 We note, by Lemma \ref{lem:UjinL2}, that $\Lambda^j\varphi$   is well-defined (at least) for all functions $\varphi\in L^2(\Sigma^j)+ L^\infty(\Sigma^j)$.
One can use \eqref{eq:hprGreen} to   write   $\Lambda^j\varphi$ explicitly, for $j=1,...,\NO$, as
\begin{multline*}
	\Lambda^j\varphi(\bsx^j)= \int_{\R}\left(\nabla\mathcal{H}(k;x_1^j-a^j,x_2^j-y_2^j)\cdot n(x_1^j,x_2^j)\right.\\\left.-\ri k \mathcal{H}(k;x_1^j-a^j,x_2^j-y_2^j)\right) \,\varphi(a^j,y_2^j)\, \rd y_2^j,\quad {\bsx}^j=(x_1^j,x_2^j) \in \Gamma^j_b.
	\end{multline*}	
	With this notation, the equality \eqref{eq:comp_robin} can be written as
	\begin{equation}\label{eq:comp_robin_bis}
		\frac{\partial u_b}{\partial n}-\ri ku_b=\Lambda^j\varphi^j\mbox{ on }\Gamma^j_b,\;j=1,...,\NO.
		\end{equation}
		Our complete formulation reads as follows:
		\begin{equation}\label{eq:HSMM_casgeneral}
		\begin{array}{c}
		\begin{array}{|lcl}
		\varphi^j=D_{j-1,j}\,\varphi^{j-1},& \text{on}\;\Sigma^j\cap\Omega^{j-1},\\
		\varphi^j=u_b\;\text{on}\;\Gamma^j,\\
		\varphi^j=D_{j+1,j}\,\varphi^{j+1},& \text{on}\;\Sigma^j\cap\Omega^{j+1},\\
		\end{array} \quad j\in \Z/\mathcal{N}\Z,\\[20pt]
		\forall v_b\in  H^1(\Omega_b),\hfill\\
		\displaystyle\int_{\Omega_b}\left(\nabla u_b\cdot \overline{\nabla v_b}-k^2 \rho u_b  \overline{v_b}\right)-\ri k\sum_{j=0}^3\int_{\Gamma^j_b} u_b  \overline{v_b} -\sum_{j=0}^3\int_{\Gamma^j_b}\Lambda^j \varphi^j\overline{v_b}=\int_{\mathcal{O}^\prime}f\overline{v_b},
		\end{array}
		\end{equation}
		where   the operators $D_{j,j\pm1}$ are defined in \eqref{eq:def_DtD_real}, and the operators $\Lambda^j$ in \eqref{eq:def_DtR}.

The following theorem gives the  well-posedness of this formulation.
\begin{theorem}\label{thm:HSMMUnique_casgeneral}
Suppose $k>0$. There exists, for each $f\in L^2(\O^\prime)$, a unique solution $(\varphi^1,..., \varphi^\NO, u_b)\in L^2_{\rm loc}(\Sigma^1)\times\ldots\times L^2_{\rm loc}(\Sigma^\NO)\times H^1(\Omega_b)$ of \eqref{eq:HSMM_casgeneral} such that, for all $j=1,...,\NO$, $\varphi^j$ satisfies the radiation condition \eqref{eq:RadCond}.
	\end{theorem}
	Existence holds by the construction above of a solution to \eqref{eq:HSMM_casgeneral} from the unique solution of (\ref{eq:helmholtz-variable}, \ref{eq:Sommerfeld}, \ref{eq:neumann}). Uniqueness follows from the following proposition since (\ref{eq:helmholtz-variable}, \ref{eq:Sommerfeld}, \ref{eq:neumann}) is well-posed.
\begin{proposition} \label{prop:HSMMUnique_casgeneral}
Suppose $k>0$ and let $(\varphi^1,..., \varphi^\NO, u_b)\in L^2_{\rm loc}(\Sigma^1)\times\ldots\times L^2_{\rm loc}(\Sigma^\NO)\times H^1(\Omega_b)$ be a solution of \eqref{eq:HSMM_casgeneral} (with $f\in L^2(\O^\prime)$) such that, for all $j=1,...,\NO$, $\varphi^j$ satisfies the radiation condition \eqref{eq:RadCond}.
Then, for all $j\in \Z/\mathcal{N}\Z$, the compatibility equations \eqref{eq:compatibility} hold   and $u_b=U^j(\varphi^j)$ in $\Omega_b\cap\Omega^j$. Moreover,
 the function  defined by
$	u=U^j(\varphi^j)$ in $\Omega^j$, for $j=1,...,\NO$, and by $u=u_b$ in $\Omega_b$
is the unique solution $u\in H^1_{\rm loc}(\Omega)$ of (\ref{eq:helmholtz-variable}, \ref{eq:Sommerfeld}, \ref{eq:neumann}). Further, $\varphi^j=u|_{\Sigma^j}$ for each $j$.
\end{proposition}

\begin{proof}
	Let $(\varphi^1,..., \varphi^\NO, u_b)\in L^2_{\rm loc}(\Sigma^1)\times\ldots\times L^2_{\rm loc}(\Sigma^\NO)\times H^1(\Omega_b)$ be a solution of \eqref{eq:HSMM_casgeneral}  such that, for all $j=1,...,\NO$, $\varphi^j$ satisfies the radiation condition \eqref{eq:RadCond},
	and let us denote by $u_\infty\in H^1_{\rm loc}(\R^2\setminus\overline{\mathcal{O}})$ the unique solution of (\ref{pb:probleme_Dir_real}, \ref{eq:Sommerfeld}) with $g=u_b|_{\partial\mathcal{O}}\in H^{1/2}(\partial\mathcal{O})$. Then as  shown in  \S \ref{sec:HSM_real}, the set of traces of $u_\infty$ on $
	\Sigma^1,\ldots,\Sigma^\NO$ solves \eqref{eq:HSMM_real}, i.e.\ the first three equations of \eqref{eq:HSMM_casgeneral}. It follows from Theorem \ref{thm:HSMMUnique} and from \eqref{eq:HSMM_casgeneral} that $\varphi^j=u_\infty|_{\Sigma^j}$ and $u_\infty=U^j(\varphi^j)$ in $\Omega^j$, for $j=1,...,\NO$. In particular,  the compatibility equations \eqref{eq:compatibility} hold and $\varphi^j=u_\infty=u_b$ on $\Gamma^j$ for $j=1,...,\NO$. Moreover, we deduce from the last equation of \eqref{eq:HSMM_casgeneral} that $-\Delta u_b-k^2\rho\, u_b=f$ in $\Omega_b$, that
	$$\frac{\partial u_b}{\partial n} = 0 \quad \mbox{on} \quad \Gamma,$$
	and that \eqref{eq:comp_robin_bis} holds. By definition \eqref{eq:def_DtR} of the DtR operators, we have then
\begin{equation} \label{eq:Lambdajeq}
		\Lambda^j\varphi^j=\left(\frac{\partial u_\infty}{\partial n}-\ri ku_\infty\right)\Big|_{\Gamma^j_b}, 
\quad j=1,...,\NO.
\end{equation}
	 Consequently, $v:=u_b-u_\infty$ belongs to $ H^1(\Omega_b\backslash \overline{\mathcal{O}})$ and satisfies
	\begin{equation} \label{pb:interiorbvp}
	\begin{array}{|lcr}
		\Delta v +  k^2v= 0 & \text{in} & \Omega_b\backslash \overline{\mathcal{O}},\\[4pt]
	v=0&\text{on}&\partial\mathcal{O},\\[4pt]
	\displaystyle \frac{\partial v}{\partial n}-\ri kv=0&\text{on}&\partial\Omega_b\backslash \Gamma,
	\end{array}
	\end{equation}
the last of these equations following from \eqref{eq:comp_robin_bis} and \eqref{eq:Lambdajeq}.
	But, for every $k>0$, this homogeneous problem has no solution except $v=0$. (To see this apply Green's identity (cf.\ \eqref{FVub1}) in $\Omega_b\backslash \overline{\mathcal{O}}$  to deduce that $\int_{\partial\Omega_b\setminus \Gamma} |v|^2 =0$, so that $v=\partial v/\partial n=0$ on $\partial\Omega_b\setminus \Gamma$, and apply Holmgren's uniqueness theorem \cite[p.~104]{ChGrLaSp:11}.) Thus $u_b=u_\infty$ in $\Omega_b\backslash \overline{\mathcal{O}}$ (in particular $u_b=u_\infty=U^j(\varphi^j)$ in $\Omega_b\cap \Omega^j$, $j=1,...,\NO$) so that the function
	$$w:=\left\{\begin{array}{lcr}
	u_b & \text{in} & \Omega_b,\\[4pt]
	u_\infty&\text{in}&\R^2\backslash \overline{\mathcal{O}},
	\end{array}\right.$$
	is well-defined, and is  the unique solution $u\in H^1_{\rm loc}(\Omega)$ of (\ref{eq:helmholtz-variable}, \ref{eq:Sommerfeld}, \ref{eq:neumann}).
	\end{proof}		

\section{Conclusion} \label{sec:conclusion}
The objective of this paper was to prove the well-posedness of the HSM formulation applied to the Helmholtz equation \eqref{eq:helmholtz-variable}  in the case of a real wavenumber $k>0$. This objective has been achieved: we have proved
  well-posedness  provided we impose the radiation condition \eqref{eq:RadCond} at infinity on the trace-unknowns $\varphi^j$ of the HSM  equations, this radiation condition analogous to the standard Sommerfeld radiation condition for the original boundary value problem. Let us recall that the results of the present paper  also complete a proof of well-posedness of the complex-scaled HSM formulation presented in \cite[\S 5]{FrenchBritish1}.

  An open question is whether the radiation condition \eqref{eq:RadCond} on the trace-unknowns, while natural, is necessary for uniqueness. The answer is not at all clear to us. But one piece of evidence that suggests that this radiation condition may not in fact be needed is that one achieves accurate results in numerical experiments by simply truncating the trace unknowns, setting, for each $j$, $\varphi^j=0$ outside some sufficiently large finite section of $\Sigma^j$, not making any use of the  radiation condition \eqref{eq:RadCond} (see \cite[Fig.~7]{BB-Fliss-Tonnoir-2018}).

Let us finish by mentioning that our main results should be extendable to more complex configurations with obstacles extending to infinity. For instance, one might consider  the general configuration of \S\ref{sec-CHSMgeneralcase} but add a screen that is a semi-infinite line $\gamma^1$,  choosing the half-planes $\Omega^j$ so that $\gamma^1$ is orthogonal to $\Sigma^1$ and $\gamma^1\cap \Omega^j=\emptyset$, for all $j\neq 1$ (see the left hand part of Figure \ref{Fig:extensions}). In this case $\Omega^1\setminus \gamma^1$ has two connected components that are quarter-planes. If $\gamma^1$ is sound soft or sound hard, i.e.\ the solution satisfies a Dirichlet or Neumann homogeneous boundary condition on $\gamma^1$,  one can derive an expression for $U^1(\varphi^1)$, the solution in $\Omega^1$ given Dirichlet data $\varphi^1$ on $\Sigma^1$, by solving separately in each quarter-plane, combining formula \eqref{eq:hprPhi} with reflection arguments.

More interesting is the case where two parallel semi-infinite lines $\gamma^j_\pm$ are introduced, so that now, for some $j$, $\Omega^j$ has three connected components, two of them quarter-planes and one of them a semi-infinite strip which is a waveguide (this is the case for $\Omega^1$ and $\Omega^3$ in the right hand part of Figure \ref{Fig:extensions}). One can derive an expression for $U^j(\varphi^j)$ in each quarter-plane as above, and an expression in the waveguide by solving as a  modal series expansion.

For each of these configurations, and for many other variations on these geometries, we can write down HSM formulations, and we expect that one should be able to prove well-posedness of the HSM formulation and equivalence with other formulations by adapting the results proved in the present paper.
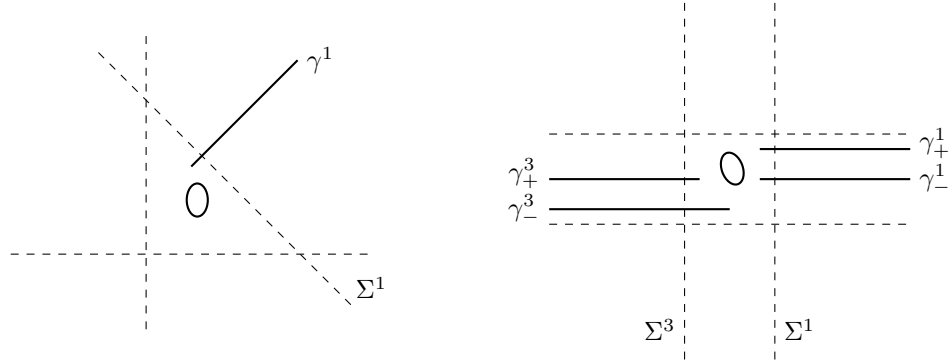
\begin{figure}[!h]
	\begin{center}
		\begin{subfigure}{.5\textwidth}
			\begin{tikzpicture}[scale=.4]
				\fill[white] (0.2,0.3)ellipse(0.35 and 0.55);
				\draw[thick] (0.2,0.3)ellipse(0.35 and 0.55);
				\draw[dashed] (-6,-1.5) -- (6,-1.5);
				\draw[dashed] (-1.5,-4) -- (-1.5,6);
				\draw[dashed,rotate=45] (1.5,-6) -- (1.5,6);
				\draw[thick,rotate=45] (6,1)--(1,1);
				\draw[right,rotate=45] (1.8,-5.5) node  {$\Sigma^1$};
				\draw[right,rotate=45] (6,1) node  {$\gamma^1$};
			\end{tikzpicture}
	
\end{subfigure}\begin{subfigure}{.5\textwidth}
		\begin{tikzpicture}[scale=.4]
		\fill[white,rotate=20] (0.2,0.3)ellipse(0.35 and 0.55);
		\draw[thick,rotate=20] (0.2,0.3)ellipse(0.35 and 0.55);
		\draw[dashed] (-6,-1.5) -- (6,-1.5);
		\draw[dashed] (-6,1.5) -- (6,1.5);
		\draw[dashed] (-1.5,-6) -- (-1.5,6);
		\draw[dashed] (1.5,-6) -- (1.5,6);
		\draw[thick] (-6,-1)--(0,-1);
		\draw[thick] (-6,0)--(-1,0);
		\draw[thick] (6,1)--(1,1);
		\draw[thick] (6,0)--(1,0);
		\draw[right] (1.5,-5) node  {$\Sigma^1$};
		\draw[left] (-1.5,-5) node  {$\Sigma^3$};
		\draw[right] (6,1.2) node  {$\gamma^1_+$};
		\draw[right] (6,0) node  {$\gamma^1_-$};
		\draw[left] (-6,-1) node  {$\gamma^3_-$};
		\draw[left] (-6,0.2) node  {$\gamma^3_+$};
	\end{tikzpicture}	
\end{subfigure}
	\end{center}
	\caption{Two geometries that could be considered.}\label{Fig:extensions}
\end{figure}

Note that if only one semi-infinite line is introduced, as in the left hand side of Figure \ref{Fig:extensions}, and no other obstacles or inhomogeneities are present, one recovers the famous Sommerfeld  half-plane problem \cite[\S38]{So:64} 
that can be solved using the Wiener-Hopf method (or other techniques). Analytical methods, notably the Wiener-Hopf method, have been extended to a variety of more complex configurations with waveguides and wedges, e.g., \cite{mittra1971analytical,LaAb:07,KiEtAl2021}, though not to all configurations that should be treatable by the extensions of the HSM method to unbounded obstacles that we suggest above.

An interesting  question is whether our uniqueness and well-posedness results for the HSM formulation can be extended to open waveguides. (The HSM method has been used to compute numerical solutions by Ott \cite{Ott} in such cases where the medium in each half-space $\Omega^j$ is stratified in the $x_2^j$ direction.) The difficulty is that the corresponding Green's function is no longer available in closed form, which makes the study of the properties of the half-plane solution more intricate. A further difference with the cases studied in this paper, and the extensions suggested above, is that, due to the possible existence of guided waves in the open waveguides, the traces $\varphi^j$ may have an oscillatory non-decaying behavior at infinity that should be taken into account in a modified radiation condition, replacing \eqref{eq:RadCond}.

\end{document}